\documentclass[11pt]{article}
\usepackage[a4paper,hscale=0.7,vscale=0.75,centering]{geometry}
\usepackage{mathtools}
\usepackage{amsfonts,amssymb}
\usepackage{amsthm}
\usepackage{mathrsfs}

\newtheorem{prop}{Proposition}
\newtheorem{thm}[prop]{Theorem}
\newtheorem{lemma}[prop]{Lemma}
\newtheorem{defi}[prop]{Definition}
\theoremstyle{remark}
\newtheorem{rmk}[prop]{Remark}

\newcommand{\enne}{\mathbb{N}}
\newcommand{\erre}{\mathbb{R}}

\newcommand{\dom}{\mathsf{D}}
\newcommand{\E}{\mathbb{E}}
\newcommand{\cF}{\mathcal{F}}
\renewcommand{\H}{\mathbb{H}}
\renewcommand{\L}{\mathbb{L}}
\newcommand{\cL}{\mathscr{L}}
\renewcommand{\P}{\mathbb{P}}

\newcommand{\embed}{\hookrightarrow}
\newcommand{\ep}[1]{{#1}^\varepsilon}
\newcommand{\lip}{\dot{C}^{0,1}}

\DeclareMathOperator{\tr}{Tr}

\DeclarePairedDelimiter\abs{\lvert}{\rvert}
\DeclarePairedDelimiter\norm{\lVert}{\rVert}
\DeclarePairedDelimiterX\ip[2]{\langle}{\rangle}{#1,#2}

\title{On well-posedness of semilinear stochastic evolution equations
  on $L_p$ spaces}

\author{Carlo Marinelli%
\thanks{Department of Mathematics, University College London, Gower
  Street, London WC1E 6BT, United Kingdom. URL: \texttt{http://goo.gl/4GKJP}}}

\date{December 14, 2015}

\begin{document}
\maketitle

\begin{abstract}
  We establish well-posedness in the mild sense for a class of
  stochastic semilinear evolution equations on $L_p$ spaces, driven by
  multiplicative Wiener noise, with a drift term given by an
  evaluation operator that is assumed to be quasi-monotone and
  polynomially growing, but not necessarily continuous. In particular,
  we consider a notion of mild solution ensuring that the evaluation
  operator applied to the solution is still function-valued, but
  satisfies only minimal integrability conditions. The proofs rely on
  stochastic calculus in Banach spaces, monotonicity and convexity
  techniques, and weak compactness in $L_1$ spaces.
\end{abstract}

\maketitle


\section{Introduction}
The purpose of this work is to prove well-posedness (existence,
uniqueness and continuous dependence of solutions on the initial
datum) to stochastic evolution equations (SEEs) of the type
\begin{equation}
  \label{eq:0}
  du(t) + Au(t)\,dt + f(u(t))\,dt = \eta u(t)\,dt + B(t,u(t))\,dW(t),
  \qquad u(0)=u_0,
\end{equation}
where $t \in [0,T]$, $A$ is a linear $m$-accretive operator on
$L_q(D)$, with $D$ a bounded domain in $\erre^n$ and $q \geq 2$,
$f:\erre\to\erre$ is an increasing function of polynomial growth
(without any continuity assumption), $W$ is a cylindrical Wiener noise
on a separable Hilbert space $H$, and $B(t,\cdot)$ is a (random) map
from $L_q(D)$ to $\cL(H,L_q(D))$ satisfying suitable Lipschitz
continuity conditions. Precise assumptions on the notion of solution
and on the data of the problem are given in Section \ref{sec:main}. In
particular, we adopt three notions of solution, that depend on the
integrability properties of $f(u)$: strict mild and mild solution are
defined to be such that $f(u) \in L_1(0,T;L_q(D))$ almost surely and
that $f(u) \in L_1(\Omega \times [0,T] \times D)$, respectively (here
$\Omega$ stands for the underlying probability space); on the other
hand, generalized solutions are defined as limits of strict mild
solutions, so that, in general, $f(u)$ may not have any
integrability. The first notion of solution is the simplest but also
the most restrictive in terms of assumptions on the data of the
problem. The second notion is the most natural if one wants $f(u)$ to
be function-valued, while satisfying minimal integrability
conditions. The last notion, motivated by analogous constructions in
the deterministic setting, apart of being the least demanding, is
useful in several contexts, for instance in the study of Kolmogorov
operators and Markovian semigroups associated to SPDEs
(cf.~e.g.~\cite{DP-K}).

Our approach to the well-posedness problem is based, on the
probabilistic side, on stochastic calculus for processes with values
in Banach spaces (of which we use only the ``simpler'' version on
spaces with type 2), and, on the analytic side, on methods from the
theory of (nonlinear) $m$-accretive operators and convex
analysis. Some ideas developed here, concerning strict mild and
generalized solutions, already appeared, in a more primitive form, in
\cite{cm:Ascona13, cm:IDAQP10} and, in a slightly different context,
in \cite{cm:EJP10}.

There is a rather large literature on semilinear dissipative SEEs,
up-to-date references to which can be found, e.g., in \cite{DPZ}. Here
we shall only discuss how our results compare to other recent ones
that are most closely related. A widely used technique to study the
well-posedness of \eqref{eq:0} consists in the reduction of the
equation to a deterministic evolution equation with random
coefficients, roughly speaking ``by subtracting the stochastic
convolution''. To the best of our knowledge, the sharpest result
obtained through this reduction is due to Barbu \cite{Barbu:Lincei},
who proved existence and uniqueness of mild solutions to \eqref{eq:0}
assuming that $q=2$, $A$ is the negative Laplacian, $B$ does not
depend on $u$ (i.e. the noise is additive), and, most importantly, the
stochastic convolution
\[
S \diamond B := t \mapsto \int_0^t S(t - s) B(s)\,dW(s),
\]
where $S$ denotes the semigroup generated by $-A$, is continuous in
time and space, and satisfies
$F(S \diamond B) \in L_1(\Omega \times [0,T] \times D)$, where $F$ is
a primitive of $f$. On the other hand, no polynomial bound on $f$ is
assumed. Our setting allows much more flexibility and no assumption is
made on the stochastic convolution, but we need an extra polynomial
growth assumption on $f$. A (partial) extension of our results to the
case of general $f$ (i.e. removing the growth assumption) and $q=2$ is
provided in a forthcoming joint work with L.~Scarpa \cite{cm:luca},
thus considerably improving on the result of \cite{Barbu:Lincei}. In
another vein, global well-posedness in the mild sense of \eqref{eq:0}
is obtained in \cite{KvN2} assuming that $S$ is an analytic semigroup
and that $f$ is polynomially bounded and locally Lipschitz continuous
on $L_q(D)$ (not as function of $\erre$!). The approach is through
approximation of the coefficients and extension of local
solutions. Even though the condition on $f$ is very restrictive,
adapting ideas from \cite{cerrai03}, and considerably improving
results thereof, well-posedness in spaces of continuous functions is
obtained, allowing $f$ to be monotone and locally Lipschitz, now only
as a function of $\erre$. Incidentally, in \cite{cerrai03}, hence also
in \cite{KvN2}, the above-mentioned reduction to a PDE with random
coefficients is again used, although in a more sophisticated
way. While the reasoning in \cite{cerrai03} relies on stochastic
calculus in Hilbert spaces and ad hoc arguments, the improvements in
\cite{KvN2} depend in an essential way on stochastic calculus in
Banach spaces. We also use techniques from this calculus (although in
a less sophisticated way), but we do not need any local
Lipschitzianity assumption, although we obviously cannot consider
solvability in spaces of continuous functions.

Our proofs do \emph{not} employ at any stage the reduction to a
deterministic equation with random coefficients. In fact, following the
classical approach of constructing solutions to regularized equations
and then passing to the limit in an appropriate topology
(cf.~e.g.~\cite{Barbu,Bre-mm} for the deterministic theory), all the
necessary estimates are obtained by stochastic calculus arguments,
rather than by classical calculus. Namely, the essential tool is
It\^o's formula for $L_q$-valued processes (even though, as explained
in Remark \ref{rmk:sconv} below, the classical formula for real
processes would suffice). Using techniques from convex analysis and
the theory of nonlinear $m$-accretive operators, we then show that,
thanks to the above-mentioned estimates, solutions to regularized
equations converge to a process that solves the original equation.

The rest of the text is organized as follows: in Section
\ref{sec:prelim} we collect several tools used in the proof of the
main results. Everything except the content of the last subsection is
known and is included here for the readers' convenience. Our main
results are stated in Section \ref{sec:main}. In Sections
\ref{sec:strict}, \ref{sec:gen}, and \ref{sec:mild} we prove
well-posedness in the strict mild, generalized, and mild sense,
respectively.

\smallskip

\noindent
\textbf{Acknowledgments.} A large part of the work for this paper was
done while the author was visiting the Interdisziplin\"ares Zentrum
f\"ur Komplexe Systeme (IZKS) at the University of Bonn. The author is
very grateful to Sergio Albeverio, his host, for the kind hospitality
and the excellent working conditions.


\section{Preliminaries}
\label{sec:prelim}
In this section we introduce notation and recall some facts that will
be used in the rest of the text. 

\smallskip

Let $(\Omega,\mathcal{F},(\mathcal{F}_t)_{0 \leq t \leq T},\P)$, with
$T>0$ fixed, be a filtered probability space satisfying the ``usual''
conditions (see e.g.~\cite{DM-mg}), and let $\E$ denote expectation
with respect to $\P$. All stochastic elements will be defined on this
stochastic basis, and any expression involving random quantities will
be meant to hold $\P$-almost surely, unless otherwise
stated. Throughout the paper, $W$ stands for a cylindrical Wiener
process on a (fixed) separable Hilbert space $H$.

Given $p>0$ and a Banach space $X$, we shall denote by $\L_p(X)$ the
set of $X$-valued random variables $\zeta$ such that
\[
\norm{\zeta}_{\L_p(X)} := \bigl(\E\norm{\zeta}_X^p\bigr)^{1/p} < \infty,
\]
and by $\H_p(X)$ the set of measurable\footnote{Since we never need
  weak measurability, measurable will always mean strongly
  measurable.}, adapted $X$-valued processes such that
\[
\norm{u}_{\H_p(X)} := 
\Bigl( \E\sup_{t \leq T} \norm{u(t)}_X^p \Bigr)^{1/p} < \infty.
\]
Both spaces are Banach spaces for $p \geq 1$, and quasi-Banach spaces
for $0<p<1$. The space $\H_p(X)$, when endowed with the equivalent
(quasi-)norm
\[
\norm{u}_{\H_{p,\alpha}(X)} := 
\Bigl( \E\sup_{t \leq T} \norm[\big]{e^{-\alpha t} u(t)}_X^p \Bigr)^{1/p},
\qquad \alpha \in \erre_+,
\]
will be denoted by $\H_{p,\alpha}(X)$.

\smallskip

The domain and range of a map $T$ will be denoted by $\mathsf{D}(T)$
and $\mathsf{R}(T)$, respectively. The standard notation $\cL(E,F)$
will be used for the space of linear bounded operators between two
Banach spaces $E$ and $F$. If $E$ and $F$ are metric spaces,
$\lip(E,F)$ stands for the set of Lipschitz maps $\phi:E \to F$ such
that
\[
\norm[\big]{\phi}_{\lip(E,F)} := \sup_{\substack{x,y\in E\\x \neq y}}
\frac{d\bigl(\phi(x),\phi(y)\bigr)}{d(x,y)} < \infty.
\]
We shall omit the indication of the spaces $E$ and $F$ when it is
clear what they are.

\smallskip

Throughout this section we shall simply write $L_q$,
$q \in [0,\infty]$, to mean the usual Lebesgue spaces over a generic
$\sigma$-finite measure space $(Y,\mathcal{A},\mu)$.

Finally, we shall use the notation $a \lesssim b$ to mean that $a$ is
less than or equal to $b$ modulo a constant, with subscripts to
emphasize its dependence on specific quantities. Completely analogous
meaning have the symbols $\gtrsim$ and $\eqsim$.

\subsection{Convex functions and subdifferentials}
\label{ssec:conv}
Let $F:\erre \to \erre \cup \{+\infty\}$ be a convex
function. Then, for any $x$, $y \in \mathsf{D}(F)$,
\begin{equation}     \label{eq:subdiff}
F(y) - F(x) \geq z (y-x)
\qquad \forall z \in \partial F(x),
\end{equation}
where $\partial F(x)$ denotes the subdifferential of $F$ at $x$. The
above inequality defines $\partial F(x)$, which is a subset of
$\erre$, in the sense that $z \in \partial F(x)$, by definition, if it
satisfies \eqref{eq:subdiff} for all $y \in \mathsf{D}(F)$.  If $F$ is
differentiable at $x \in \erre$, then $\partial F(x)$ reduces to a
singleton and coincides with $F'(x)$. The following mean-value theorem
holds (cf.~\cite[Theorem~2.3.4, p.~179]{lema}): if $F$ is
finite-valued, one has, for any $x$, $y \in \erre$,
\[
F(y) - F(x) = \int_x^y s(r)\,dr,
\]
where $s(r)$ is \emph{any} selection of the subdifferential
$\partial F(r)$.

Given a maximal monotone graph $f \subset \erre^2$ (see
{\S}\ref{ssec:m-accr} below), there exists a convex function $F$,
called the potential of $f$, such that $f = \partial F$. The converse
is also true, i.e. the map $x \mapsto \partial F(x)$ defines a maximal
monotone graph of $\erre^2$ for any convex function $F$.

The (Legendre-Fenchel) conjugate $F^*:\erre \to \erre \cup \{+\infty\}$
of the convex (proper, lower semicontinuous) function
$F:\erre \to \erre \cup \{+\infty\}$ is defined as
\[
F^*(x) := \sup_{y \in \mathsf{D}(F)} \big(xy - F(y)\bigr).
\]
$F^*$ is itself a convex (proper, lower semicontinuous) function. The
definition obviously implies $xy \leq F(x) + F^*(y)$ for all $x$,
$y \in \erre$, with equality if and only if $y \in \partial F(x)$,
which in turn is equivalent to $x \in \partial F^*(y)$.
Moreover, if $F$ is everywhere finite on $\erre$, then $F^*$ is
superlinear at infinity, i.e.
\[
\lim_{\abs{x}\to\infty} \frac{F^*(x)}{\abs{x}} = +\infty.
\]
In particular, if $\partial F(x) \neq \varnothing$ for all
$x \in \erre$, or, equivalently, the domain of $f:=\partial F$ is
$\erre$, then $F$ is finite-valued on $\erre$ and $F^*$ is superlinear
at infinity (see, e.g., \cite[Chapter~E]{lema} for all these facts).

\subsection{Duality mapping and differentiability of the norm}
Let $X$ be a Banach space with (topological) dual $X^*$. The duality
mapping of $X$ is the map
\begin{align*}
J: X &\to 2^{X^*}\\
   x &\mapsto \bigl\{ x^* \in X^*:\, \ip{x^*}{x} = \norm{x}_X^2 
= \norm{x^*}^2_{X^*} \bigr\}.
\end{align*}
If $X^*$ is strictly convex, then $J$ is single-valued and continuous
from $X$, endowed with the strong topology, to $X^*$, endowed with the
weak topology (i.e. $J$ is demicontinuous). Moreover, if $X^*$ is
uniformly convex, then $J$ is uniformly continuous on bounded subsets
of $X$.  For instance, all Hilbert spaces and all $L_q$ spaces with
$1 < q < \infty$ are uniformly convex (hence also strictly convex),
and their duality mappings are single-valued and demicontinuous. In
particular, if $X=L_q$, $1 < q < \infty$, one has
\[
J: u \mapsto \norm{u}_{L_q}^{2-q} \abs{u}^{q-2} u.
\]
On the other hand, the duality mapping of $L_1$ is multivalued: in
fact, if $X=L_1$, one has
\[
J: u \mapsto \bigl\{ v \in L_\infty:\, v \in \norm{u}_{L_1} \operatorname{sgn} u
\;\; \text{a.e.} \bigr\}.
\]
Moreover, one has $J = \partial\phi$, where
$\phi=\frac12 \norm{\cdot}^2_X$ and $\partial$ stands for the
subdifferential in the sense of convex analysis.
The ($q$-th power of the) norm of $L_q$ spaces with $q \geq 2$ is in
fact very regular: setting $\Phi_q := \norm{\cdot}_{L_q}^q$, one has
$\Phi_q \in C^2(L_q)$, with
\begin{gather*}
\Phi_q': L_q \to \cL(L_q,\erre) \simeq L_{q'},
\qquad
\Phi_q'': L_q \to \cL(L_q,\cL(L_q,\erre)) \simeq \cL_2(L_q),\\
\begin{split}
\Phi_q'(u): v &\longmapsto q \ip[\big]{\abs{u}^{q-2}u}{v} \equiv
q \int_Y \abs{u}^{q-2}uv\,d\mu,\\
\Phi_q''(u): (v,w) &\longmapsto q(q-1) \ip[\big]{\abs{u}^{q-2}v}{w} \equiv
q(q-1) \int_Y \abs{u}^{q-2}vw\,d\mu,
\end{split}
\end{gather*}
where $\cL_2(L_q)$ stands for the space of bilinear forms on $L_q$.
In particular, for any $u \in L_q$,
\begin{equation}\label{eq:nJ}
  \Phi_q'(u) = q \norm{u}_{L_q}^{q-2} J(u), \qquad
  \norm[\big]{\Phi'(x)}_{L_{q'}} = q \norm[\big]{x}_{L_{q}}^{q-1},
\end{equation}
and, by H\"older's inequality,
\begin{equation}
  \label{eq:n2}
  \norm{\Phi_q''(u)}_{\cL_2(L_q)} \leq q(q-1) \norm{u}^{q-2}_{L_q}.
\end{equation}
A detailed treatment of duality mappings and related geometric
properties of $L_q$ spaces can be found, for instance, in
\cite{Cioranescu}, while most results needed here are also recalled
in, e.g., \cite[Chapter~1]{Barbu}.

\subsection{$m$-accretive operators}     \label{ssec:m-accr}
A subset $A$ of $X \times X$ is called \emph{accretive} if, for every
$(x_1,y_1)$, $(x_2,y_2) \in A$, there exists $z \in J(x_1-x_2)$ such
that $\ip{y_1-y_2}{z} \geq 0$. An accretive set $A$ is called
$m$-accretive if $\mathsf{R}(I+A)=X$. One often says that $A$ is a
multivalued (nonlinear) mapping on $X$, rather than a subset of $X
\times X$.  Through the rest of this subsection, we shall assume that
$A$ is an $m$-accretive subset of $X \times X$.

The \emph{Yosida approximation} (or \emph{regularization}) of $A$ is
the family $\{A_\lambda\}_{\lambda>0}$ of (single-valued) operators on
$X$ defined by
\[
A_\lambda := \frac{1}{\lambda} \bigl(I-(I+\lambda A)^{-1}\bigr), 
\qquad \lambda>0.
\]
The following properties will be extensively used:
\begin{itemize}
\item[(a)] $A_\lambda$ is $m$-accretive;
\item[(b)] $A_\lambda \in \lip(X,X)$ with Lipschitz constant bounded
  by $2/\lambda$;
\item[(c)] $\norm{A_\lambda x} \leq \inf_{y \in Ax} \norm{y}$ for all
  $x \in X$;
\item[(d)] $A_\lambda x \in A(I+\lambda A)^{-1} x$ for all $x \in X$;
\item[(d')] if $A$ is single-valued and $X$, $X^*$ are uniformly
  convex, then $A_\lambda x \to Ax$ as $\lambda \to 0$ for all $x \in
  \dom(A)$.
\item[(e)] $(I+\lambda A)^{-1} \in \lip(X,X)$ with Lipschitz constant bounded
  by $1$;
\item[(f)] $(I+\lambda A)^{-1}x \to x$ as $\lambda \to 0$ for all $x
  \in \overline{\dom(A)}$.
\end{itemize}
If $X^*$ is uniformly convex, then the $m$-accretive set $A$ is
demiclosed, i.e. it is closed in $X \times X_w$, where $X_w$ stands
for $X$ endowed with its weak topology. More precisely, if $x_n \to x$
strongly in $X$ and $A_{\lambda_n}x_n \to y$ weakly in $X$ as $n \to
\infty$, then $(x,y) \in A$.

Let $X=L_q$, $1 \leq q < \infty$. If $g$ is a maximal
monotone graph in $\erre^2$, then the (multivalued) evaluation
operator $\overline{g}$ associated to $g$ is an $m$-accretive subset
of $X \times X$. The operator $\overline{g}$ is defined on $X$ as
\[
\overline{g}: u \mapsto \bigl\{ v \in X:\; v \in g(u) \quad \mu\text{-a.e.}
\bigr\}.
\]
Note that the graph of a (discontinuous) increasing function
$g_0:\erre \to \erre$ is a monotone subset of $\erre^2$, but it is not
maximal monotone. However, the graph $g \subset \erre^2$ defined by
\[
g(x)=
\begin{cases}
  g_0(x), & x \in \erre \setminus I,\\
  [g_0(x-), g_0(x+)], & x \in I,
\end{cases}
\]
where $I$ is the jump set of $g_0$, is maximal monotone and (clearly)
extends $g_0$. We shall not explicitly distinguish below among an
increasing function $g_0$, its maximal monotone extension $g$, and the
associated evaluation operator $\overline{g}$.

The proofs of the above facts (and much more) can be found, for
instance, in \cite[\S 2.3]{Barbu}.\footnote{Formula (3.12) in
  \emph{op.~cit.} contains a misprint: $A_{\lambda_n}x$ should be
  replaced by $A_{\lambda_n}x_n$.}

\subsection{$\gamma$-Radonifying operators}
We shall use only basic facts from the rich and powerful theory of
$\gamma$-Radonifying operators. For more information we refer to,
e.g., the survey \cite{vN-surv}.

Let $H$, $H'$ be real separable Hilbert spaces and $X$, $X'$ Banach
spaces. An operator $T \in \cL(H,X)$ is said to be
$\gamma$-Radonifying if there exists an orthonormal basis
$(h_n)_{n\in\enne}$ of $H$ such that
\[
\norm{T}_{\gamma(H,X)} := \biggl( \E'\norm[\Big]{%
\sum_{n\in\enne} \gamma_n Th_n} \biggr)^{1/2} < \infty,
\]
where $(\gamma_n)$ is a sequence of independent identically
distributed standard Gaussian random variable on a probability space
$(\Omega',\cF',\P')$. One can shows that $\norm{T}_{\gamma(H,X)}$ does
not depend on the choice of the orthonormal basis $(h_n)$. The set of
all $T \in \cL(H,X)$ such that $\norm{T}_{\gamma(H,X)}$ is finite is
itself a Banach space with norm $\norm{\cdot}_{\gamma(H,X)}$, and
$\gamma(H,X)$ is a two-sided ideal of $\cL(H,X)$, i.e.
$L \in \cL(X,X')$ and $R \in \cL(H',H)$ imply
\[
\norm{LTR}_{\gamma(H',X')} \leq 
\norm{L}_{\cL(X,X')} \norm{T}_{\gamma(H,X)} \norm{R}_{\cL(H',H)}.
\]
The following convergence result is a simple corollary of the ideal
property: if $L_n \to L$ strongly in $\cL(X,X')$ as $n \to \infty$,
i.e. $L_nx \to Lx$ in $X'$ for all $x \in X$, then
\[
\norm[\big]{L_nT-LT}_{\gamma(H,X')} \xrightarrow{n\to\infty} 0.
\]
If $X=L_q$, $q \geq 1$, by a simple application of the Khinchin-Kahane
inequalities it follows that $T \in \gamma(H,L_q)$ if and only if
\[
\norm[\big]{(Th_n)}_{L_q(\ell_2)} = \norm[\bigg]{%
\Bigl( \sum_{n\in\enne} \abs{Th_n}^2 \Bigr)^{1/2}}_{L_q} < \infty
\]
for all orthonormal bases $(h_n)$ of $H$, and
$\norm{T}_{\gamma(H,L_q)} \eqsim \norm{(Th_n)}_{L_q(\ell_2)}$. Moreover, the mapping
\begin{align*}
  L_q(H) &\longrightarrow \gamma(H,L_q)\\
  f &\longmapsto T_f: g \mapsto \ip{f(\cdot)}{g}_H
\end{align*}
is an isomorphism of Banach spaces, where one can take 
\[
f = \norm{(Th_n)}_{\ell_2} = \Bigl( \sum_{n\in\enne} \abs{Th_n}^2 \Bigr)^{1/2}.
\]

\subsection{Stochastic calculus in Banach spaces}
Let $X$ be a UMD Banach space. For any $1<p<\infty$ and progressively
measurable process $G \in \L_p(\gamma(L_2((0,T);H),X))$, the
stochastic integral of $G$ with respect to $W$ is a well-defined
$X$-valued local martingale that satisfies Burkholder inequality
\[
\E\sup_{t\leq T} \norm[\bigg]{\int_0^t G(s)\,dW(s)}^p_X \eqsim_{p,X}
\E\norm[\big]{G}^p_{\gamma(L_2((0,T);H),X)}.
\]
If $X$ has type 2 (this is the case if $X=L_q$, $q \geq 2$), one has
the continuous embedding\footnote{If $X=L_q$, $q \geq 2$, the only
  case of interest for us, the embedding is just an obvious
  consequence of Minkowski's inequality:
  $L_2(0,t;\gamma(H,L_q)) \simeq L_2(0,t;L_q(H)) \embed
  L_q(L_2(0,t;H)) \simeq \gamma(L_2((0,t);H),L_q)$.}
\[
L_2(0,T;\gamma(H,X)) \embed \gamma(L_2((0,T);H),X),
\]
hence
\begin{equation}
  \label{eq:pippo}
  \E\sup_{t \leq T} \norm[\bigg]{\int_0^t G(s)\,dW(s)}^p_X
  \lesssim_{p,X}
  \E \biggl( \int_0^T \norm[\big]{G(s)}^2_{\gamma(H,X)}\,ds\biggr)^{p/2}
\end{equation}
for all $p>0$ (the case $0<p\leq 1$ follows by Lenglart's domination
inequality, see~\cite{Lenglart}). Note that, if $X=L_q$, in view of
the isomorphism mentioned at the end of last subsection, the above
inequalities can be equivalently written only in terms of $L_q$
norms. In other words, for our purposes the use of
$\gamma$-Radonifying norms amounts only to adopting a convenient
language. For further details we refer to \cite{vNVW:integ} and
references therein.

\smallskip

We shall also need It\^o's formula for $L_q$-valued processes,
and we use the version of \cite{Brz:Ito}, which is valid for
UMD-valued processes. For our purposes, however, previous less general
versions (cited in \cite{Brz:Ito}) would also do, as well as the very
specific one of \cite{Krylov:ItoLp}, where only the $q$-th power
of the norm is considered.
Let us first introduce some notation: if $\Phi \in \cL_2(X)$ is a
bilinear form on $X$ and $T \in \gamma(H,X)$, we set
\[
\tr_T \Phi := \sum_{n\in\enne} \Phi(Th_n,Th_n),
\]
for which it is easily seen that
\begin{equation}
  \label{eq:tr-ineq}
  \abs{\tr_T \Phi} \leq \norm{\Phi}_{\cL_2(X)} \norm{T}^2_{\gamma(H,X)}. 
\end{equation}

\begin{thm}     \label{thm:Ito}
  Let $X$ be a UMD Banach space, and consider the $X$-valued process
  \[
  u(t) = u_0 + \int_0^t b(s)\,ds + \int_0^t G(s)\,dW(s),
  \]
  where
  \begin{itemize}
  \item[\emph{(a)}] $u_0:\Omega \to X$ is
    $\cF_0$-measurable;
  \item[\emph{(b)}] $b:\Omega \times [0,T] \to X$ is measurable, adapted and
    such that $b \in L_1(0,T;X)$;
  \item[\emph{(c)}] $G:\Omega \times [0,T] \to \cL(H,X)$ is
    $H$-measurable, adapted, stochastically integrable with respect to
    $W$, and such that $G \in L_2(0,T;\gamma(H,X))$.
  \end{itemize}
  For any $\varphi \in C^2(X)$, one has
  \begin{align*}
    \varphi(u(t)) &= \varphi(u_0) + \int_0^t \varphi'(u(s))b(s)\,ds
    + \int_0^t \varphi'(u(s))G(s)\,dW(s)\\
    &\quad + \frac12 \int_0^t \tr_{G(s)} \varphi''(u(s))\,ds.
  \end{align*}
\end{thm}

\subsection{Estimates for linear equations}
Given a Banach space $X$ and a linear $m$-accretive operator $A$ on
$X$, for any $X$-valued or $\cL(H,X)$-mapping $h$, we shall write, for
any $\varepsilon>0$, $\ep{h}:=(I+\varepsilon A)^{-1}h$.

We first prove an estimate that will be used repeatedly in the
following.
\begin{prop}
  \label{prop:gf}
  Let $A$ be a linear $m$-accretive operator on $L_q$, and consider
  the unique mild solution $u$ to the equation
  \[
  du(t) + Au(t) = b(t)\,dt + G(t)\,dW(t), \qquad u(0)=u_0,
  \]
  where $u_0$, $b$, and $G$ satisfy the assumptions of Theorem
  \ref{thm:Ito} (with $X=L_q)$. If $u \in L_\infty(L_q)$, then
  \begin{align*}
  \norm[\big]{u(t)}_{L_q}^q &\leq \norm[\big]{u_0}_{L_q}^q 
  + \int_0^t \Phi'_q(u(s))b(s)\,ds
  + \int_0^t \Phi'_q(u(s))G(s)\,dW(s)\\
  &\quad + \frac12 q(q-1) \int_0^t \norm[\big]{G(s)}^2_{\gamma(H,L_q)}%
  \norm[\big]{u(s)}_{L_q}^{q-2}\,ds
\end{align*}
\end{prop}
\begin{proof}
  It is not difficult to verify that $\ep{u}$ is the unique strong
  solution to
  \[
  d\ep{u} + A\ep{u} = \ep{b}\,dt + \ep{G}\,dW,
  \qquad \ep{u}(0)=\ep{u}_0
  \]
  (cf.~e.g.~\cite[Lemma~6]{cm:SIMA12}). It\^o's formula then
  yields\footnote{From now on we shall occasionally omit the
    indication of the time parameter, if no confusion may arise, for
    notational compactness.}
  \begin{align*}
    \norm[\big]{\ep{u}(t)}_{L_q}^q + \int_0^t \Phi'_q(\ep{u}) A\ep{u}\,ds &=
    \norm[\big]{\ep{u}_0}_{L_q}^q + \int_0^t \Phi'_q(\ep{u}) \ep{b}\,ds
    +  \int_0^t \Phi'_q(\ep{u}) \ep{G}\,dW\\
    &\quad + \frac12 \int_0^t \tr_{\ep{G}} \Phi''_q(\ep{u})\,ds,
  \end{align*}
  where
  $\Phi'_q(\ep{u}) A\ep{u} = q \norm[\big]{\ep{u}}_{L_q}^{q-2}
  \ip[\big]{A\ep{u}}{J(\ep{u})} \geq 0$
  by accretivity of $A$ on $L_q$, and
  $\norm[\big]{\ep{u}_0}_{L_q} \leq \norm[\big]{u_0}_{L_q}$ by
  contractivity of $(I+\varepsilon A)^{-1}$ on $L_q$. We are thus left
  with
\begin{align*}
    \norm[\big]{\ep{u}(t)}_{L_q}^q &\leq
    \norm[\big]{u_0}_{L_q}^q + \int_0^t \Phi'_q(\ep{u}) \ep{b}\,ds
    +  \int_0^t \Phi'_q(\ep{u}) \ep{G}\,dW\\
    &\quad + \frac12 \int_0^t \tr_{\ep{G}} \Phi''_q(\ep{u})\,ds.
  \end{align*}
  We are now going to pass to the limit as $\varepsilon \to 0$ in this
  inequality. One clearly has
  $\norm[\big]{\ep{u}(t)}_{L_q} \to \norm[\big]{u(t)}_{L_q}$ as
  $\varepsilon \to 0$ because $(I+\varepsilon A)^{-1}$ converges
  strongly to the identity in $\cL(L_q)$ as $\varepsilon \to 0$.
  By the triangle inequality,
  \begin{align*}
    \sup_{t \leq T} \abs[\bigg]{\int_0^t \Phi'_q(\ep{u}) \ep{b}\,ds
    - \int_0^t \Phi'_q(u)b\,ds} &\leq
    \int_0^T \abs[\big]{\bigl(\Phi'_q(\ep{u})-\Phi'_q(u)\bigr)\ep{b}}\,ds\\
    &\quad + \int_0^T \abs[\big]{\Phi'_q(u)(\ep{b}-b)}\,ds.
  \end{align*}  
  The following reasoning is to be understood to hold for each fixed
  $\omega$ in a subset of $\Omega$ of full $\P$-measure. Since
  $\ep{u}(s) \to u(s)$ and $\ep{b}(s) \to b(s)$ in $L_q$, hence also
  in measure, for all $s\in[0,T]$, and $\Phi_q'$ is continuous, it
  follows that
  \[
  \abs[\big]{\bigl(\Phi'_q(\ep{u}(s))-\Phi'_q(u(s))\bigr)\ep{b}(s)}
  \xrightarrow{\varepsilon \to 0} 0
  \]
  in measure for all $s$. Moreover,
  \begin{align*}
  \abs[\big]{\bigl(\Phi'_q(\ep{u}(s))-\Phi'_q(u(s))\bigr)\ep{b}(s)} &\leq
  \norm[\big]{\Phi'_q(\ep{u}(s))-\Phi'_q(u(s))}_{L_{q'}} \,
  \norm[\big]{b(s)}_{L_q}\\
  &\lesssim \norm[\big]{u(s)}_{L_q}^{q-1}  \, \norm[\big]{b(s)}_{L_q}
  \leq \norm[\big]{u}_{L_\infty(L_q)}^{q-1}\, \norm[\big]{b(s)}_{L_q}
  \end{align*}
  and
  $\norm[\big]{u}_{L_\infty(L_q)}^{q-1}\, \norm[\big]{b}_{L_q} \in
  L_1(0,T)$, which imply, by the dominated convergence theorem,
  \[
  \int_0^T
  \abs[\big]{\bigl(\Phi'_q(\ep{u})-\Phi'_q(u)\bigr)\ep{b}}\,ds
  \xrightarrow{\varepsilon \to 0} 0.
  \]
  By a completely analogous argument one shows that
  $\int_0^T \abs[\big]{\Phi'_q(u)(\ep{b}-b)}\,ds
  \xrightarrow{\varepsilon \to 0} 0$, hence also that
  \[
  \int_0^t \Phi'_q(\ep{u}) \ep{b}\,ds \xrightarrow{\varepsilon \to 0}
  \int_0^t \Phi'_q(u)b\,ds.
  \]
  Let us now show that
  \[
  M_\varepsilon(t) :=  \int_0^t \Phi'_q(\ep{u})\ep{G}\,dW 
  \xrightarrow{\varepsilon \to 0} M(t) :=
  \int_0^t \Phi'_q(u) G\,dW
  \]
  in probability. Recall that, for the sequence of continuous local
  martingales $(M_\varepsilon-M)$, one has
  $\sup_{t \leq T} \abs{M_\varepsilon(t)-M(t)} \to 0$ in probability if and
  only if $[M_\varepsilon-M,M_\varepsilon-M](T) \to 0$ in probability (see
  e.g. \cite[Proposition~17.6]{kall}). We have
  \[
  [M_\varepsilon-M,M_\varepsilon-M](T) = \int_0^T
  \norm[\big]{\Phi'_q(\ep{u})\ep{G} - \Phi'_q(u)G}^2_{\gamma(H,\erre)}\,ds,
  \]
  and, by the triangle inequality,
  \begin{multline*}
  \norm[\big]{\Phi'_q(\ep{u})\ep{G} - \Phi'_q(u)G}_{\gamma(H,\erre)}\\
  \leq \norm[\big]{\Phi'_q(\ep{u})-\Phi'_q(u)}_{L_{q'}}
  \norm[\big]{\ep{G}}_{\gamma(H,L_q)} +
  \norm[\big]{\Phi'_q(u)}_{L_{q'}}
  \norm[\big]{\ep{G}-G}_{\gamma(H,L_q)},
  \end{multline*}
  where
  \[
  \norm[\big]{\Phi'_q(\ep{u})-\Phi'_q(u)}_{L_{q'}}
  \norm[\big]{\ep{G}}_{\gamma(H,L_q)} \leq
  \norm[\big]{\Phi'_q(\ep{u})-\Phi'_q(u)}_{L_{q'}}
  \norm[\big]{G}_{\gamma(H,L_q)} \xrightarrow{\varepsilon \to 0} 0
  \]
  pointwise in the time variable, and
  $\norm{\ep{G}-G}_{\gamma(H,L_q)} \to 0$
  because $(I+\varepsilon A)^{-1}$ converges strongly to the identity
  in $\cL(L_q)$.  The above also yields
  \[
  \norm[\big]{\Phi'_q(\ep{u})\ep{G} - \Phi'_q(u)G}_{\gamma(H,\erre)} \lesssim
  \norm{u}_{L_q}^{q-1} \norm{G}_{\gamma(H,L_q)},
  \]
  where, since $G \in L_2(0,T;\gamma(H,L_q))$ and $u \in L_\infty(L_q)$,
  \[
  \int_0^T \norm{u(s)}^{2(q-1)}_{L_q} \norm{G(s)}^2_{\gamma(H,L_q)}\,ds \leq
  \norm{u}^{2(q-1)}_{L_\infty(L_q)} \int_0^T \norm{G(s)}^2_{\gamma(H,L_q)}\,ds
  < \infty.
  \]
  Therefore, by the dominated convergence theorem,
  \[
  [M_\varepsilon-M,M_\varepsilon-M](T) \xrightarrow{\varepsilon \to 0} 0
  \]
  in probability.
  Finally, by \eqref{eq:n2}, \eqref{eq:tr-ineq}, and the ideal
  property of $\gamma(H,L_q)$,
  \begin{align*}
  \int_0^t \tr_{\ep{G}(s)} \Phi''_q(\ep{u}(s))\,ds &\leq
  q(q-1) \int_0^t \norm[\big]{\ep{G}(s)}^2_{\gamma(H,L_q)}
     \norm[\big]{\ep{u}(s)}_{L_q}^{q-2}\,ds\\
  &\leq q(q-1) \int_0^t \norm[\big]{G(s)}^2_{\gamma(H,L_q)}
     \norm[\big]{u(s)}_{L_q}^{q-2}\,ds.
  \qedhere
  \end{align*}
\end{proof}

We now establish a maximal inequality for stochastic convolutions that
might be interesting in its own right (see Remark \ref{rmk:sconv}
below).  We shall use the following notation, already used in the
Introduction:
\[
S \diamond G (t) := \int_0^t S(t-s)G(s)\,dW(s).
\]
\begin{thm}\label{thm:sconv}
  Let $p>0$ and $q \geq 2$. If $G$ satisfies the hypothesis of
  Theorem~\ref{thm:Ito}, then the stochastic convolution
  $S \diamond G$ has (a modification with) continuous paths and
  \[
  \E\sup_{t\leq T} \norm[\big]{S \diamond G(t)}^p_{L_q}
  \lesssim \E \biggl( \int_0^T \norm[\big]{G(s)}^2_{\gamma(H,L_q)}\,ds
  \biggr)^{p/2}
  \]
\end{thm}
\begin{proof}
  We proceed in two steps, first assuming that $G$ takes values in
  $\mathsf{D}(A)$, then removing this assumption.
  \smallskip\par\noindent
  \textsc{Step 1.} Let us assume for the moment that
  $G \in L_2(0,T;\gamma(H,\mathsf{D}(A)))$. As in the proof of
  Proposition \ref{prop:gf}, it is easy to see that $S \diamond G$ is
  the unique strong solution to
  \[
  du(t) + Au(t)\,dt = G(t)\,dW(t), \qquad u(0)=0.
  \]
  Then It\^o's formula yields
  \[
  \norm[\big]{u(t)}_{L_q}^q + \int_0^t \Phi'_q(u) Au\,ds =
  \int_0^t \Phi'_q(u) G\,dW
  + \frac12 \int_0^t \tr_G \Phi''_q(u)\,ds.
  \]
  Setting
  \[
  v := \norm[\big]{u}_{L_q}^q, \qquad
  b := \frac12 \tr_G \Phi''_q(u) - \Phi'_q(u) Au, \qquad
  g := \Phi'_q(u) G,
  \]
  we can write
  \[
  v(t) = \int_0^t b(s)\,ds + \int_0^t g(s)\,dW(s).
  \]
  Let $\alpha \geq 1$ be arbitrary but fixed. Then
  $\varphi:x \mapsto x^{2\alpha} \in C^2$ with
  \[
  \varphi'(x) = 2\alpha x^{2\alpha-1}, \qquad
  \varphi''(x) = 2\alpha (2\alpha-1) x^{2(\alpha-1)}.
  \]
  Therefore, by It\^o's formula for real processes,
  \begin{align*}
  \norm[\big]{u(t)}_{L_q}^{2\alpha q} = \varphi(v(t)) &= 
  \int_0^t \Bigl( \varphi'(v(s)) b(s) 
  + \frac12 \varphi''(v(s)) \norm{g(s)}^2_{\gamma(H,\erre)}
    \Bigr)\,ds\\
  &\quad + \int_0^t \varphi'(v(s)) g(s)\,dW(s),
  \end{align*}
  where, by the accretivity of $A$, \eqref{eq:n2} and \eqref{eq:tr-ineq},
  \begin{align*}
  \int_0^t \varphi'(v(s)) b(s)\,ds &= 
  \int_0^t \norm[\big]{u(s)}_{L_q}^{(2\alpha-1)q}
  \Bigl( \frac12 \tr_G \Phi''_q(u(s)) - \Phi'_q(u(s))Au(s) \Bigr)\,ds\\
  &\lesssim \int_0^T \norm[\big]{u(s)}_{L_q}^{2(\alpha q-1)}
  \norm[\big]{G(s)}^2_{\gamma(H,L_q)}\,ds\\
  &\leq \norm[\big]{u}_{L_\infty(L_q)}^{2(\alpha q-1)}
  \int_0^T \norm[\big]{G(s)}^2_{\gamma(H,L_q)}\,ds\\
  &\leq \varepsilon \norm[\big]{u}_{L_\infty(L_q)}^{2\alpha q}
  + N(\varepsilon) \biggl( 
    \int_0^T \norm[\big]{G(s)}^2_{\gamma(H,L_q)}\,ds \biggr)^{\alpha q},
  \end{align*}
  where we have applied Young's inequality\footnote{From now on, whenever
    we apply Young's inequality, we shall mostly state only the exponents
    used.} in the form
  \[
  xy \leq \varepsilon x^{\frac{\alpha q}{\alpha q-1}}
  + N(\varepsilon) y^{\alpha q}
  \qquad \forall x,y \geq 0, \; \varepsilon>0.
  \]
  Similarly,
  \[
  \int_0^t \varphi''(v(s)) \norm{g(s)}^2_{\gamma(H,\erre)}\,ds
  \lesssim \int_0^t \norm[\big]{u(s)}_{L_q}^{2(\alpha-1)q}
  \norm[\big]{g(s)}^2_{\gamma(H,\erre)}\,ds,
  \]
  where, by the ideal property of $\gamma$-Radonifying operators and
  \eqref{eq:nJ},
  \[
  \norm[\big]{g}_{\gamma(H,\erre)} =
  \norm[\big]{\Phi'_q(u) G}_{\gamma(H,\erre)}
  \leq \norm[\big]{\Phi'_q(u)}_{\cL(L_q)} \norm[\big]{G}_{\gamma(H,L_q)}
  \lesssim \norm[\big]{u}^{q-1}_{L_q} \norm[\big]{G}_{\gamma(H,L_q)},
  \]
  hence, proceeding exactly as before,
  \begin{align*}
  \int_0^t \varphi''(v(s)) \norm{g(s)}^2_{\gamma(H,\erre)}\,ds
  &\lesssim \int_0^t \norm[\big]{u(s)}_{L_q}^{2(\alpha q-1)}
  \norm[\big]{G(s)}^2_{\gamma(H,L_q)}\,ds\\
  &\leq \varepsilon \norm[\big]{u}_{L_\infty(L_q)}^{2\alpha q}
  + N(\varepsilon) \biggl( 
    \int_0^T \norm[\big]{G(s)}^2_{\gamma(H,L_q)}\,ds \biggr)^{\alpha q}.
  \end{align*}
  Finally, Davis' inequality yields
  \[
  \E\sup_{t\leq T} \abs[\bigg]{\int_0^t \varphi'(v(s)) g(s)\,dW(s)}
  \lesssim \E \biggl( \int_0^T
  \norm[\big]{\varphi'(v(s)) g(s)}^2_{\gamma(H,\erre)}\,ds \biggr)^{1/2},
  \]
  where
  \[
  \norm[\big]{\varphi'(v) g}_{\gamma(H,\erre)} \lesssim
  \norm[\big]{u}_{L_q}^{2\alpha q - 1} \norm[\big]{G}_{\gamma(H,L_q)},
  \]
  which implies, by Young's inequality with exponents
  $2\alpha q/(2\alpha q-1)$ and $2\alpha q$,
  \begin{align*}
  \biggl(
  \int_0^T \norm[\big]{\varphi'(v(s)) g(s)}^2_{\gamma(H,\erre)}\,ds
  \biggr)^{1/2} &\lesssim \norm[\big]{u}_{L_\infty(L_q)}^{2\alpha q - 1}
  \biggl(\int_0^T \norm[\big]{G(s)}^2_{\gamma(H,L_q)}\,ds\biggr)^{1/2}\\
  &\leq \varepsilon \norm[\big]{u}_{L_\infty(L_q)}^{2\alpha q}
  + N(\varepsilon) \biggl(
  \int_0^T \norm[\big]{G(s)}^2_{\gamma(H,L_q)}\,ds \biggr)^{\alpha q}.
  \end{align*}
  Taking $\varepsilon$ small enough, the claim is proved in the case
  $p \geq 2q$. The case $0<p<2q$ follows by Lenglart's domination
  inequality (see \cite{Lenglart}).
  \smallskip\par\noindent
  \textsc{Step 2.} Recall that $\ep{G}:=(I+\varepsilon A)^{-1}G \to G$
  in $L_2(0,T;\gamma(H,L_q))$ as $\varepsilon \to 0$ and
  $\ep{G} \in L_2(0,T;\gamma(H,\mathsf{D}(A)))$, and
  $\ep{u}:=S \diamond \ep{G} = \ep{(S \diamond G)}$ is the unique
  strong solution to
  \[
  d\ep{u}(t) + A\ep{u}(t)\,dt = \ep{G}(t)\,dW(t), \qquad \ep{u}(0)=0.
  \]
  It is elementary to show, by the previous step, that $(\ep{u})$ is a
  Cauchy sequence in $\H_p(L_q)$, and that its limit is a modification
  of $S \diamond G$. Since $\ep{u}$ has continuous paths and the convergence in
  $\H_p(L_q)$ implies almost sure uniform convergence of paths, we
  conclude that $u$ has a modification with continuous paths.
\end{proof}
\begin{rmk}
  \label{rmk:sconv}
  The previous Theorem is actually a special case of \cite{vNZ}, who
  considered the case of $X$-valued stochastic convolutions, with $X$
  a $2$-smooth Banach space. Our proof, although similar in spirit
  (the idea, as a matter of fact, goes back at least to
  \cite{Tubaro}), is interesting in the sense that it does \emph{not}
  use any infinite-dimensional calculus. To wit, It\^o's formula for the
  $q$-th power of the $L_q$-norm reduces to nothing else than the
  one-dimensional It\^o formula and Fubini's theorem (cf. the proof in
  \cite{Krylov:ItoLp}).
\end{rmk}


\section{Main results}
\label{sec:main} 
Let $D$ be an open bounded subset of $\erre^n$ with smooth boundary.
All Lebesgue spaces on $D$ will be denoted without explicit mention of
the domain, e.g. $L_q:=L_q(D)$. The mixed-norm spaces
$L_p(0,T;L_q(D))$ will simply be denoted by $L_p(L_q)$. Denoting the
Lebesgue measure on $[0,T]$ and on $D$ by $dt$ and $dx$, respectively,
we define the measure $m:=\P \otimes dt \otimes dx$ on
$\Omega \times [0,T] \times D$.

To look for $L_q$-valued (mild) solutions to \eqref{eq:0}, it is clear
that the linear operator $A$ should be taken as the generator of a
$C_0$-semigroup on $L_q$, and that the map $B:L_q \to \cL(H,L_q)$
should satisfy suitable Lipschitz continuity assumptions. For later
use, we introduce the following conditions, where $r>0$, $s \geq 2$:
\begin{itemize}
\item[($\mathrm{A}_s$)] $A$ is a linear $m$-accretive operator on $L_s$.
\item[($\mathrm{B}_{r,s}$)] The map
  $B:\Omega \times [0,T] \times L_s \to \gamma(H,L_s)$ is such that
  $B(\cdot,\cdot,x)$ is $H$-measurable and adapted for all
  $x \in L_s$, there is a constant $\norm{B}_{\lip}$ such that
  \[
  \norm[\big]{B(\omega,t,u) - B(\omega,t,v)}_{\gamma(H,L_s)} \leq
  \norm{B}_{\lip} \norm[\big]{u - v}_{L_s} \qquad \forall (\omega,t)
  \in \Omega \times [0,T],
  \]
  and $B(\cdot,\cdot,0) \in \L_r(L_2(0,T;\gamma(H,L_s)))$.
\end{itemize}
If $A$ satisfies ($\mathrm{A}_s$), the $C_0$-semigroup of contractions
generated by $-A$ on $L_s$ will be denoted by $S$. Should $A$ satisfy
($\mathrm{A}_s$) for different values of $s$, we shall not
notationally distinguish among different (but consistent)
realizations of $A$ and $S$ on different $L_s$ spaces.

We assume that the function $f:\erre \to \erre$ is increasing and that
there exists $d \geq 1$ such that $\abs{f(x)} \lesssim 1 + \abs{x}^d$ for
all $x \in \erre$. In particular, $f$ is \emph{not} assumed to be
continuous. We shall denote the maximal monotone graph associated to
$f$ (see {\S}\ref{ssec:m-accr}) by the same symbol.

\begin{rmk}
  Thanks to the linear term in $u$ on the right-hand side of
  \eqref{eq:0}, nothing changes assuming that $f$ (or $A$, or both) is
  quasi-monotone, i.e. that $f+\delta I$ is monotone for some
  $\delta>0$.
\end{rmk}

We shall establish well-posedness of \eqref{eq:0} in several classes
of processes. The most natural, and most restrictive, notion of
solution is the following.
\begin{defi}
  Let $u_0$ be an $L_q$-valued $\mathcal{F}_0$-measurable random
  variable. A measurable adapted $L_q$-valued processes $u$
  is a \emph{strict mild solution} to \eqref{eq:0} if
  $u \in L_\infty(L_q)$, there exists an adapted $L_q$-valued process
  $g \in L_1(L_q)$, with $g \in f(u)$ $m$-a.e., and, for all
  $t \in [0,T]$, $S(t-\cdot)B(\cdot,u)$ is stochastically integrable and
  \begin{equation}
  \label{eq:mild}
    u(t) + \int_0^t S(t-s)\bigl( g(s) - \eta u(s)\bigr)\,ds 
    = S(t)u_0 + \int_0^t S(t-s)B(s,u(s))\,dW(s).
  \end{equation}
\end{defi}

Our first main result provides sufficient conditions for the
well-posedness of \eqref{eq:0} in $\H_p(L_q)$. The proof is given in
Section \ref{sec:strict} below.
\begin{thm}     
  \label{thm:main0}
  Let $p>0$ and $q \geq 2$ be such that
  \[
  p^* := \frac{p}{q} (2d+q-2) > d.
  \]
  Assume that
  \begin{itemize}
  \item[\emph{(a)}] $u_0 \in \L_{p^*}(L_{qd})$;
  \item[\emph{(b)}] hypothesis $(\mathrm{A}_s)$ is
  satisfied for $s=q$ and $s=qd$;
  \item[\emph{(c)}] hypothesis $(\mathrm{B}_{r,s})$ is
  satisfied for $r=p$, $s=q$ and $r=p^*$, $s=qd$.
  \end{itemize}
  Then there exists a unique strict mild solution $u \in \H_p(L_q)$ to
  \eqref{eq:0}. Moreover, $u$ has continuous paths and the solution
  map $u_0 \mapsto u$ is Lipschitz continuous from $\L_p(L_q)$ to
  $\H_p(L_q)$.
\end{thm}

Relaxing the definition of solution, well-posedness for \eqref{eq:0}
can be proved for \emph{any} $p>0$ and $q \geq 2$.
The following notion of solution derives from the definition of
\emph{solution faible} by Benilan and Br\'ezis \cite{BenBre}. We first
deal with with equations with additive noise, i.e. of the type
\begin{equation}
\label{eq:add}
du(t) + Au(t)\,dt + f(u(t))\,dt = \eta u(t)\,dt + B(t)\,dW(t),
\qquad u(0)=u_0,
\end{equation}
where $B \in \L_p(L_2(0,T;\gamma(H,L_q)))$.
\begin{defi}
  Let $p>0$ and $q \geq 2$. A process $u \in \mathbb{H}_p(L_q)$ is a
  \emph{generalized solution} to \eqref{eq:add} if there exist sequences
  $(u_{0n})_n \subset \L_p(L_q)$,
  $(B_n)_n \subset \L_p(L_2(0,T;\gamma(H,L_q)))$, and
  $(u_n)_n \subset \mathbb{H}_p(L_q)$ such that $u_{0n} \to u_0$ in
  $\L_p(L_q)$, $B_n \to B$ in $\L_p(L_2(0,T;\gamma(H,L_q)))$, and
  $u_n \to u$ in $\mathbb{H}_p(L_q)$ as $n \to \infty$, where $u_n$ is the
  (unique) strict mild solution to
  \[
  du_n(t) + Au_n(t)\,dt + f(u_n(t))\,dt = \eta u_n(t)\,dt + B_n(t)\,dW(t),
  \qquad u_n(0)=u_{0n}.
  \]
\end{defi}

\begin{thm}
  \label{thm:ga}
  Let $p>0$ and $q \geq 2$. Assume that
  \begin{itemize}
  \item[\emph{(a)}] $u_0 \in \L_p(L_q)$;
  \item[\emph{(b)}] hypothesis $(\mathrm{A}_s)$ is satisfied for $s=q$
    and $s=qd$;
  \item[\emph{(c)}]
    $B \in \L_p(L_2(0,T;\gamma(H,L_q)))$.
  \end{itemize}
  Then \eqref{eq:add} admits a unique generalized solution
  $u \in \H_p(L_q)$. Moreover, $u$ has continuous paths and the
  solution map $u_0 \mapsto u$ is Lipschitz continuous from
  $\L_p(L_q)$ to $\H_p(L_q)$.
\end{thm}
In order to define generalized solutions to \eqref{eq:0} we need some
preparations. In particular, we (formally) introduce the map $\Gamma$
on $\H_p(L_q)$ defined by $\Gamma: v \mapsto w$, where $w$ is the
unique generalized solution to
\[
dw(t) + Aw(t)\,dt + f(w(t))\,dt = \eta w(t)\,dt + B(v(t))\,dW(t),
\qquad u(0)=u_0,
\]
if it exists. Otherwise, if no generalized solution exists, then we
set $\Gamma(v)=\varnothing$.
\begin{defi}
  A process $u \in \H_p(L_q)$ is a generalized solution to
  equation \eqref{eq:0} if it is a fixed point of $\Gamma$ in
  $\H_p(L_q)$.
\end{defi}
\begin{thm}     
  \label{thm:gm}
  Let $p>0$ and $q \geq 2$. Assume that
  \begin{itemize}
  \item[\emph{(a)}] $u_0 \in \L_p(L_q)$;
  \item[\emph{(b)}] hypothesis $(\mathrm{A}_s)$ is satisfied for $s=q$
    and $s=qd$;
  \item[\emph{(c)}] hypothesis $(\mathrm{B}_{p,q})$ is satisfied.
  \end{itemize}
  Then \eqref{eq:0} admits a unique generalized solution
  $u \in \H_p(L_q)$. Moreover, $u$ has continuous paths and the
  solution map $u_0 \mapsto u$ is Lipschitz continuous from
  $\L_p(L_q)$ to $\H_p(L_q)$.
\end{thm}
The proofs of Theorems \ref{thm:ga} and \ref{thm:gm} are given in
Section \ref{sec:gen} below. Note that, if $u \in \H_p(L_q)$ is a
generalized solution to \eqref{eq:0}, we cannot claim that $f(u)$
admits a selection $g$ such that $\int_0^t S(t-s)g(s)\,ds$ in
\eqref{eq:mild} is well defined, essentially because we do not have
enough integrability for $g$.

Under additional assumptions we obtain existence of a (unique)
solution $u$ for which $f(u)$ admits a selection $g$ satisfying the
``minimal'' integrability condition $g \in \L_1(L_1(L_1))$.
\begin{defi}
  \label{def:mild}
  Let $u_0$ be an $L_q$-valued $\cF_0$-measurable random variable. A
  measurable adapted $L_q$-valued process
  $u \in L_\infty(L_q)$ is a \emph{mild solution} to equation
  \eqref{eq:0} if there exists $g \in \L_1(L_1(L_1))$, with
  $g \in f(u)$ $m$-a.e., and, for all $t \in [0,T]$,
  $S(t-\cdot)B(\cdot,u)$ is stochastically integrable and
  \eqref{eq:mild} is satisfied for all $t \in [0,T]$.
\end{defi}
The corresponding well-posedness result holds in a subset of
$\H_p(L_q)$ defined in terms of the potential of $f$, for which we
assume that $0 \in f(0)$. We need some definitions first: for
$q \geq 2$, let $\phi_q$ be the homeomorphism of $\erre$ defined by
$\phi_q: x \mapsto x \abs{x}^{q-2}$, and $F:\erre \to \erre$ be the
potential of $f$, i.e. a convex function such that $\partial F = f$,
which we ``normalize'' so that $F(0)=0$ (in particular $F \geq
0$).
Similarly, setting $\tilde{f} := f \circ \phi_q^{-1}$, $\tilde{F}$
stands for the potential of $\tilde{f}$, subject to the same
normalization, and $\tilde{F}^*$ for its Legendre-Fenchel
conjugate. Finally, we set $\hat{F}:=\tilde{F} \circ \phi_q$. A simple
computation shows that the convex function $\hat{F}$ is the potential
of the maximal monotone graph $x \mapsto f(x) \phi_q'(x)$.
\begin{thm}
\label{thm:mild}
  Let $p \geq q \geq 2$. Assume that
  \begin{itemize}
  \item[\emph{(a)}] $u_0 \in \L_p(L_q)$;
  \item[\emph{(b)}] hypothesis $(\mathrm{A}_s)$ is satisfied for
    $s \in \{1, q, qd\}$;
  \item[\emph{(b')}] the resolvent $R_\lambda:=(I+\lambda A)^{-1}$,
    $\lambda>0$, is positivity preserving and such that
    $R_\lambda^{\sigma}(L_1) \subset L_q$ for some $\sigma\in\enne$;
  \item[\emph{(c)}] hypothesis $(\mathrm{B}_{p,q})$ is satisfied;
  \item[\emph{(d)}] $0 \in f(0)$ and $F$ is even.
  \end{itemize}
  Then \eqref{eq:0} admits a unique mild solution $u \in \H_p(L_q)$
  such that $\hat{F}(u)$, $\tilde{F}^*(g) \in
  \L_1(L_1(L_1))$.
  Moreover, $u$ has continuous paths and the solution map
  $u_0 \mapsto u$ is Lipschitz continuous from $\L_p(L_q)$ to
  $\H_p(L_q)$.
\end{thm}
The proof is given in Section \ref{sec:mild} below. It should be
remarked that unconditional well-posedness in $\H_p(L_q)$, i.e.
without any further conditions on $u$, remains an open problem.

Hypothesis (b) is satisfied by large classes of operators, for
instance all generators of sub-Markovian semigroups on $L_1(D)$ and of
symmetric semigroups on $L_2(D)$. Their resolvent are also positivity
preserving. The ``hypercontractivity'' of the resolvent in hypothesis
(b') is satisfied, for example, by non-degenerate second order
elliptic operators, under very mild regularity assumptions on the
coefficients, thanks to elliptic regularity results and Sobolev
embedding theorems.


\section{Strict mild solutions}
\label{sec:strict}
Consider the regularized equation
\begin{equation}
  \label{eq:reg}
  du_\lambda(t) + Au_\lambda(t)\,dt + f_\lambda(u_\lambda(t))\,dt
  = \eta u_\lambda(t)\,dt + B(u_\lambda(t))\,dW(t),
  \qquad u_\lambda(0)=u_0,
\end{equation}
where $f_\lambda:\erre \to \erre$, $\lambda>0$, is the Yosida
approximation of $f$, so that $f_\lambda$ is Lipschitz continuous on
$\erre$, as well as on $L_q$ (when viewed as evaluation operator). As
is natural to expect, well-posedness of \eqref{eq:reg} in the strict
mild sense holds in $\H_p(L_q)$ for all $p>0$ and $q \geq 2$.
\begin{prop}
  Let $p>0$ and $q \geq 2$. Assume that hypotheses $(\mathrm{A}_q)$
  and $(\mathrm{B}_{p,q})$ are verified. If $f:\erre \to \erre$ is
  Lipschitz continuous and $u_0 \in \L_p(L_q)$, then the equation
  \[
  du(t) + Au(t)\,dt + f(u(t))\,dt = B(u(t))\,dW(t),
  \qquad u(0)=u_0,
  \]
  admits a unique strict mild solution $u \in \H_p(L_q)$ with
  continuous paths, and the solution map $u_0 \mapsto u$ is Lipschitz
  continuous from $\L_p(L_q)$ to $\H_p(L_q)$.
\end{prop}
\begin{proof}
  Since the proof proceeds by the classical fixed point argument, we
  omit some simple details. Consider the map, formally defined for the
  moment,
  \[
  \Gamma: (u_0,u) \mapsto S(t)u_0 - \int_0^t S(t-s) f(u(s))\,ds
  + \int_0^t S(t-s)B(u(s))\,dW(s).
  \]
  To prove existence and uniqueness, it suffices to show that
  $\Gamma(u_0,\cdot)$ is an everywhere defined contraction on
  $\H_p(L_q)$ for any $u_0 \in \L_p(L_q)$. One has
  \begin{align*}
  &\sup_{t\leq T} \norm[\bigg]{%
  \int_0^t S(t-s)\bigl(f(u(s))-f(v(s))\bigr)\,ds}_{L_q}\\
  &\hspace{5em} \leq \int_0^T \norm[\big]{f(u(s))-f(v(s))}_{L_q}\,ds
  \leq T \norm{f}_{\lip} \sup_{t\leq T} \norm[\big]{u(t)-v(t)}_{L_q},
  \end{align*}
  hence, writing $S \ast f$ to denote the second term in the
  above definition of $\Gamma$,
  \[
  \norm[\big]{S \ast (f(u)-f(v))}_{\H_p(L_q)} \leq
  T \norm{f}_{\lip} \norm[\big]{u-v}_{\H_p(L_q)}.
  \]
  Similarly, it follows by Theorem \ref{thm:sconv} that
  \begin{align*}
  &\E\sup_{t\leq T} \norm[\bigg]{%
  \int_0^t S(t-s)\bigl(B(u(s))-B(v(s))\bigr)\,dW(s)}_{L_q}^p\\
  &\hspace{5em} \lesssim \E\biggl( \int_0^T
  \norm[\big]{B(u(s)-B(v(s))}_{\gamma(H,L_q)}^2\,ds \biggr)^{p/2}\\
  &\hspace{5em} \leq T^{p/2} \norm{B}^p_{\lip}
  \norm[\big]{u-v}^p_{\H_p(L_q)},
  \end{align*}
  i.e., for a constant $N$ independent of $T$,
  \[
  \norm[\big]{S \diamond (B(u)-B(v))}_{\H_p(L_q)} \leq
  N T^{1/2} \, \norm{B}_{\lip} \norm[\big]{u-v}_{\H_p(L_q)}.
  \]
  Therefore, choosing $T$ small enough, one finds a constant
  $c \in \mathopen]0,1\mathclose[$ such that
  \[
  \norm[\big]{\Gamma(u_0,u) - \Gamma(u_0,v)}_{\H_p(L_q)} \leq c 
  \norm[\big]{u - v}_{\H_p(L_q)}.
  \]
  It is clear that $\Gamma(\H_p(L_q)) \subset \H_p(L_q)$. Recalling
  that the function
  \begin{align*}
  d: \H_p(L_q) \times \H_p(L_q) &\to \erre_+\\
  (x,y) &\mapsto \norm[\big]{x-y}^{1 \wedge p}_{\H_p(L_q)}
  \end{align*}
  is a metric on $\H_p(L_q)$, Banach's contraction principle yields
  the existence of a unique fixed point of $\Gamma$ on the complete
  metric space $(\H_p(L_q),d)$, which is the unique strict mild
  solution we are looking for on the interval $[0,T]$. Writing
  $u=\Gamma(u_0,u)$, $v=\Gamma(v_0,v)$, the Lipschitz continuity of
  the solution map follows by $c<1$ and
  \begin{align*}
  \norm[\big]{u-v}_{\H_p(L_q)} &= 
  \norm[\big]{\Gamma(u_0,u) - \Gamma(v_0,v)}_{\H_p(L_q)}\\
  &\leq \norm[\big]{\Gamma(u_0,u) - \Gamma(u_0,v)}_{\H_p(L_q)}
  + \norm[\big]{\Gamma(u_0,v) - \Gamma(v_0,v)}_{\H_p(L_q)}\\
  &\leq c \norm[\big]{u-v}_{\H_p(L_q)} + \norm[\big]{u_0-v_0}_{\L_p(L_q)}.
  \end{align*}
  By a classical patching argument, the smallness restriction on $T$
  can be removed. Continuity of paths follows by Theorem
  \ref{thm:sconv}.
\end{proof}
\begin{rmk}
  Even though quite sophisticated well-posedness results exist for
  SEEs on $L_q$ spaces with Lipschitz continuous coefficients
  (cf.~e.g.~\cite{Brz,vNVW}), the previous simple Proposition does not
  seem to follow from the existing literature. For instance, in
  \emph{op.~cit.} the semigroup $S$ is assumed to be analytic (but not
  necessarily accretive), and (in \cite{vNVW}) solutions are sought in
  spaces strictly contained in $\H_p(L_q)$, and $p>2$. It may indeed
  be possible to deduce the above well-posedness result from
  \emph{op.~cit.}, but it seems much easier to give a direct
  proof.
\end{rmk}

We now proceed to considering equation \eqref{eq:0}. In this section
we first show that a priori estimates on $u_\lambda$ imply
well-posedness of equation \eqref{eq:0}, then obtain such estimates
(under additional assumptions on $A$ and $B$), thus proving Theorem
\ref{thm:main0}. Our argument depends on passing to the limit as
$\lambda \to 0$ in the mild form of the regularized equation
\eqref{eq:reg}.

\subsection{A priori estimates imply well-posedness}
We begin establishing sufficient conditions for $(u_\lambda)_\lambda$
to be a Cauchy sequence in $\H_p(L_q)$, whose limit is then a natural
candidate as solution to \eqref{eq:0}.
\begin{lemma}     \label{lm:Cauchy}
  Let $p > 0$, $q \geq 2$, $p^* := p(2d+q-2)/q$, and assume that
  hypotheses $(\mathrm{A}_q)$ and $(\mathrm{B}_{p,q})$ are satisfied.
  If the sequence $(u_\lambda)$ is bounded in
  $\H_{p^*}(L_{2d+q-2})$, then $(u_\lambda)$ is a Cauchy sequence
  in $\mathbb{H}_p(L_q)$.
\end{lemma}
\begin{proof}
  Let us define, for a constant parameter $\alpha>\eta$ to be
  chosen later, $v_\lambda(t) := e^{-\alpha t} u_\lambda(t)$ for
  all $t \geq 0$, so that
  \[
  dv_\lambda(t) = -\alpha v_\lambda(t) + e^{-\alpha t} u_\lambda(t),
  \]
  hence also, for $\mu>0$,
  \begin{equation}
  \label{eq:diff}
  \begin{split}
  d(v_\lambda-v_\mu) + \Bigl( (\alpha-\eta)(v_\lambda-v_\mu) &+ A(v_\lambda-v_\mu)
  + e^{-\alpha t} \bigl(f_\lambda(u_\lambda) - f_\mu(u_\mu)\bigr)\Bigr)\,dt\\
  &= e^{-\alpha t} \bigl(B(u_\lambda)-B(u_\mu)\bigr)\,dW,
  \end{split}
  \end{equation}
  in the (strict) mild sense, with initial condition
  $v_\lambda(0)-v_\mu(0)=0$.  Proposition \ref{prop:gf} yields
  \begin{equation}\label{eq:stima}
  \begin{split}
  &\norm[\big]{v_\lambda(t)-v_\mu(t)}_{L_{q}}^{q} 
  + q(\alpha-\eta) \int_0^t \norm{v_\lambda-v_\mu}_{L_q}^q\,ds\\
  &\hspace{5em}\qquad + \int_0^t e^{-\alpha s} \Phi'_{q}(v_\lambda-v_\mu) 
             \bigl(f_\lambda(u_\lambda) - f_\mu(u_\mu)\bigr)\,ds\\
  &\hspace{5em} \leq \int_0^t e^{-\alpha s} \Phi'_{q}(v_\lambda - v_\mu)
    \bigl(B(s,u_\lambda) - B(s,u_\mu)\bigr)\,dW\\
  &\hspace{5em}\qquad + \frac12 q(q-1) \int_0^t
  \norm[\big]{e^{-\alpha s}\bigl(B(u_\lambda)-B(u_\mu)\bigr)}_{\gamma(H,L_q)}^2
           \norm[\big]{v_\lambda-v_\mu}_{L_q}^{q-2}\,ds.
  \end{split}
  \end{equation}
  We are going to estimate each term appearing in this inequality.
  Note that $\Phi'_q(cx) = c^{q-1}\Phi_q'(x)$ for all $c \in \erre_+$
  and $x \in L_q$, hence
  \[
  \Phi'_{q}(v_\lambda-v_\mu) 
             \bigl(f_\lambda(u_\lambda) - f_\mu(u_\mu)\bigr) =
  e^{-\alpha(q-1)s} \, \Phi'_{q}(u_\lambda-u_\mu) 
             \bigl(f_\lambda(u_\lambda) - f_\mu(u_\mu)\bigr)
  \]
  and
  \[
  \Phi'_{q}(u_\lambda-u_\mu) \bigl(f_\lambda(u_\lambda) - f_\mu(u_\mu)\bigr)
  = q \int_D \abs{u_\lambda-u_\mu}^{q-2} (u_\lambda - u_\mu)
  \bigl(f_\lambda(u_\lambda) - f_\mu(u_\mu)\bigr)\,dx,
  \]
  where, setting $J_\lambda:=(I+\lambda f)^{-1}$, $\lambda>0$, and
  writing
  \begin{align*}
  u_{\lambda} - u_{\mu} &= u_{\lambda} - J_\lambda u_{\lambda}
  + J_\lambda u_{\lambda} - J_\mu u_{\mu}
  + J_\mu u_{\mu} - u_{\mu}\\
  &= \lambda f_\lambda(u_{\lambda}) + J_\lambda u_{\lambda}
     - J_\mu u_{\mu} - \mu f_\mu(u_{\mu}),
  \end{align*}
  one has, by monotonicity of $f$ and recalling that
  $f_\lambda = f \circ J_\lambda$,
  \begin{align*}
  \bigl( f_\lambda(u_{\lambda})-f_\mu(u_{\mu}) \bigr) (u_{\lambda}-u_{\mu})
  &\geq \bigl( f_\lambda(u_{\lambda}) - f_\mu(u_{\mu}) \bigr)
  \bigl( \lambda f_\lambda(u_{\lambda}) - \mu f_\mu(u_{\mu}) \bigr)\\
  &\geq \lambda \abs[\big]{f_\lambda(u_{\lambda})}^2 
  + \mu \abs[\big]{f_\mu(u_{\mu})}^2 - (\lambda + \mu) 
    \abs[\big]{f_\lambda(u_{\lambda})} \abs[\big]{f_\mu(u_{\mu})}\\
  &\geq -\frac{\mu}{2} \abs[\big]{f_\lambda(u_{\lambda})}^2
   -\frac{\lambda}{2} \abs[\big]{f_\mu(u_{\mu})}^2\\
  &\geq -\frac12 (\lambda+\mu) \big( \abs[\big]{f_\lambda(u_{\lambda})}^2
   + \abs[\big]{f_\mu(u_{\mu})}^2 \bigr).
  \end{align*}
  Moreover, since
  $\abs{f_\lambda(x)} \leq \abs{f(x)} \lesssim 1 + \abs{x}^d$ for all
  $x \in \erre$ and
  $\abs{x-y}^{q-2} \lesssim \bigl(\abs{x}+\abs{y}\bigr)^{q-2}$ for all
  $x$, $y \in \erre$ (the latter inequality holds because $q \geq 2$),
  one infers
  \begin{align*}
  &\bigl( f_\lambda(u_{\lambda})-f_\mu(u_{\mu}) \bigr) (u_{\lambda}-u_{\mu})
  \abs{u_\lambda - u_\mu}^{q-2}\\
  &\hspace{7em} \gtrsim -(\lambda + \mu)
  \bigl( 1 + \abs{u_\lambda}^{2d} + \abs{u_\mu}^{2d} \bigr)
  \abs{u_\lambda - u_\mu}^{q-2}\\
  &\hspace{7em} \gtrsim -(\lambda+\mu) \Bigl( 1 + 
   \bigl( \abs{u_\lambda} + \abs{u_\mu} \bigr)^{2d} \Bigr)
   \bigl( \abs{u_\lambda} + \abs{u_\mu} \bigr)^{q-2}\\
  &\hspace{7em} \gtrsim -(\lambda+\mu) \Bigl( 1 + 
   \abs{u_\lambda}^{2d+q-2} + \abs{u_\mu}^{2d+q-2} \Bigr),
  \end{align*}
  thus also
  \begin{align*}
  &\int_0^t e^{-\alpha s} \Phi'_{q}(v_\lambda-v_\mu) \bigl(f_\lambda(u_\lambda) -
  f_\mu(u_\mu)\bigr)\,ds\\
  &\hspace{7em} \gtrsim - (\lambda+\mu) \int_0^t e^{-q\alpha s} \Bigl( 1 +
  \norm[\big]{u_\lambda}_{L_{2d+q-2}}^{2d+q-2} +
  \norm[\big]{u_\mu}_{L_{2d+q-2}}^{2d+q-2} \Bigr)\,ds\\
  &\hspace{7em} \gtrsim - (\lambda+\mu) \, \frac{1-e^{-q\alpha t}}{q\alpha}
  \Bigl( 1 + \sup_{s\leq t} \norm[\big]{u_\lambda(s)}_{L_{2d+q-2}}^{2d+q-2} +
  \sup_{s\leq t} \norm[\big]{u_\mu(s)}_{L_{2d+q-2}}^{2d+q-2} \Bigr),
  \end{align*}
  which estimates the third term on the left-hand side of
  \eqref{eq:stima}.

  Since $\Phi''_q(cx) = c^{q-2}\Phi''_q(x)$ for all $c \in \erre_+$
  and $x \in L_q$, recalling \eqref{eq:n2} and \eqref{eq:tr-ineq}, the
  Lipschitz continuity of $B$ implies that the integrand in the last
  term on the right-hand side of \eqref{eq:stima} is estimated by
  \begin{multline*}
    e^{-q\alpha s} \norm[\big]{u_\lambda-u_\mu}_{L_{q}}^{q-2}
    \norm[\big]{B(u_\lambda)-B(u_\mu)}^2_{\gamma(H,L_{q})}\\
    \leq \norm[\big]{B}^2_{\lip} e^{-q\alpha s}
    \norm[\big]{u_\lambda-u_\mu}_{L_{q}}^{q} = \norm[\big]{B}^2_{\lip}
    \norm[\big]{v_\lambda-v_\mu}_{L_{q}}^{q}.
  \end{multline*}
  In particular, collecting the second term on the right-hand side and
  the second term on the left-hand side of \eqref{eq:stima}, we obtain
  \begin{align*}
  \norm[\big]{v_\lambda(t)-v_\mu(t)}_{L_{q}}^{q} 
  &+ q\bigl(\alpha - \eta - \norm[\big]{B}^2_{\lip} (q-1)/2 \bigr) 
     \int_0^t \norm{v_\lambda-v_\mu}_{L_q}^q\,ds\\
  &\lesssim (\lambda + \mu) \, \frac{1-e^{-q\alpha t}}{q\alpha}
   \Bigl( 1 + \sup_{s\leq t} \norm[\big]{u_\lambda(s)}_{L_{2d+q-2}}^{2d+q-2} +
  \sup_{s\leq t} \norm[\big]{u_\mu(s)}_{L_{2d+q-2}}^{2d+q-2} \Bigr)\\
  &\quad + \abs[\bigg]{\int_0^t e^{-\alpha s} \Phi'_{q}(v_\lambda - v_\mu)
    \bigl(B(u_\lambda) - B(u_\mu)\bigr)\,dW}.
  \end{align*}
  Raising both sides to the power $p/q$, taking supremum in
  time,\footnote{Note that $A(t)+B(t) \leq C(t)$ for all $t$, with $A$,
    $B$, $C$ positive functions of $t$, implies
    $\sup_t A(t) \leq \sup_t C(t)$ and $\sup_t B(t) \leq \sup_t C(t)$,
    hence $\sup_t A(t) + \sup_t B(t) \leq 2\sup_t C(t)$.} %
  then expectation, one gets, setting $p^*:=(2d+q-2)p/q$,
  \begin{equation}
    \label{eq:stronza}
  \begin{split}
  &\E\sup_{t\leq T} \norm[\big]{v_\lambda(t)-v_\mu(t)}_{L_q}^p\\
  &\hspace{2em}
   + q^{p/q}
     \Bigl( \alpha - \eta - \frac12 \norm[\big]{B}^2_{\lip} (q-1) \Bigr)^{p/q}
     \E\biggl( \int_0^{T} \norm[\big]{v_\lambda-v_\mu}^q_{L_q}\,ds
       \biggr)^{p/q}\\
  &\hspace{5em} \lesssim
  (\lambda+\mu)^{p/q} \,
  \Bigl( 1 + \E\sup_{t \leq T} \norm[\big]{u_\lambda(t)}_{L_{2d+q-2}}^{p^*}
  + \E\sup_{t \leq T} \norm[\big]{u_\mu(t)}_{L_{2d+q-2}}^{p^*} \Bigr)\\
  &\hspace{5em}\quad + \E\sup_{t\leq T} \abs[\bigg]{%
   \int_0^t e^{-\alpha s} \Phi'_{q}(v_\lambda - v_\mu)
    \bigl(B(u_\lambda) - B(u_\mu)\bigr)\,dW}^{p/q},
  \end{split}
  \end{equation}
  where, by the Burkholder-Davis-Gundy inequality,
  \begin{align*}
  &\E\sup_{t\leq T} \abs[\bigg]{%
   \int_0^t e^{-\alpha s} \Phi'_{q}(v_\lambda - v_\mu)
    \bigl(B(u_\lambda) - B(u_\mu)\bigr)\,dW}^{p/q}\\
  &\hspace{5em} \lesssim
  \E\biggl( \int_0^{T} \norm[\big]{%
    e^{-\alpha s} \Phi'_{q}(v_\lambda - v_\mu) 
    \bigl(B(u_\lambda)-B(u_\mu)\bigr)}^2_{\gamma(H,\erre)}\,ds 
  \biggr)^{\frac{p}{2q}}.
  \end{align*}
  Thanks to the ideal property of $\gamma$-Radonifying operators,
  identity \eqref{eq:nJ}, and the Lipschitz continuity of $B$, one has
  \begin{align*}
    &\int_0^{T} 
    \norm[\big]{e^{-\alpha s} \Phi'_{q}(v_\lambda-v_\mu)%
         \bigl(B(u_\lambda)-B(u_\mu)\bigr)}^2_{\gamma(H,\erre)}\,ds\\
    &\hspace{5em} \leq \int_0^{T} \norm[\big]{\Phi'_{q}(v_\lambda-v_\mu)}^2_{L_{q'}}
    \norm[\big]{e^{-\alpha s}\bigl(B(u_{\lambda})-B(u_\mu)\bigr)}^2_{\gamma(H,L_{q})}\,ds\\
    &\hspace{5em} \lesssim \int_0^{T} \norm[\big]{v_\lambda-v_\mu}^{2(q-1)}_{L_{q}}
    \norm[\big]{e^{-\alpha s}\bigl(B(u_{\lambda})-B(u_\mu)\bigr)}^2_{\gamma(H,L_{q})}\,ds\\
    &\hspace{5em} \leq \norm[\big]{B}^2_{\lip} \sup_{t\leq T}
            \norm[\big]{v_\lambda-v_\mu}^{2(q-1)}_{L_q}
            \int_0^{T} \norm[\big]{v_{\lambda}-v_\mu}^2_{L_q}\,ds.
  \end{align*}
  This implies
  \begin{align*}
    &\E\sup_{t\leq T} \abs[\bigg]{%
    \int_0^t e^{-\alpha s} \Phi'_{q}(v_\lambda - v_\mu)
    \bigl(B(u_\lambda) - B(u_\mu)\bigr)\,dW}^{p/q}\\
    &\hspace{3em} \lesssim \norm[\big]{B}_{\lip}^{p/q}
    \E\sup_{t \leq T} \norm[\big]{v_\lambda-v_\mu}^{p(q-1)/q}_{L_{q}}
     \biggl(\int_0^{T} \norm[\big]{v_{\lambda}-v_\mu)}^2_{L_q}\,ds
     \biggr)^{\frac{p}{2q}}\\
    &\hspace{3em} \leq \varepsilon \norm[\big]{B}_{\lip}^{p/q}
    \E\sup_{t \leq T} \norm[\big]{v_\lambda-v_\mu}^{p}_{L_{q}}
    + N_1(\varepsilon) \norm[\big]{B}_{\lip}^{p/q}
    \E\biggl(\int_0^{T} \norm[\big]{v_{\lambda}-v_\mu)}^2_{L_q}\,ds
      \biggr)^{\frac{p}{2}}\\
    &\hspace{3em} \leq \varepsilon \norm[\big]{B}_{\lip}^{p/q}
    \E\sup_{t \leq T} \norm[\big]{v_\lambda-v_\mu}^{p}_{L_{q}}
    + N_1(\varepsilon) \norm[\big]{B}_{\lip}^{p/q} T^{1-\frac2q}
    \E\biggl(\int_0^{T} \norm[\big]{v_{\lambda}-v_\mu)}^q_{L_q}\,ds
      \biggr)^{\frac{p}{q}},
  \end{align*}
  where we have used Young's inequality with exponents $q/(q-1)$ and
  $q$ in the second-last step, and H\"older's inequality with
  exponents $q/2$ and $q/(q-2)$ in the last step (recall that
  $q \geq 2$).  By \eqref{eq:stronza}, we conclude that there exist
  constants $N_2$, $N_3$, independent of $\lambda$, $\mu$ and
  $\alpha$, with $N_2$ also independent of $\varepsilon$, such that
  \begin{align*}
    &\E\sup_{t\leq T} \norm[\big]{v_\lambda(t)-v_\mu(t)}_{L_q}^p
    + q^{p/q} \Bigl( 
      \alpha - \eta - \frac12 (q-1) \norm[\big]{B}^2_{\lip} \Bigr)^{p/q}
      \E\biggl( \int_0^{T} \norm[\big]{v_\lambda-v_\mu}^q_{L_q}\,ds
       \biggr)^{p/q}\\
    &\hspace{5em}\leq \varepsilon N_2 \E\sup_{t\leq T}
    \norm[\big]{v_\lambda(t)-v_\mu(t)}_{L_q}^p
    + N_3 \E\biggl( \int_0^{T} \norm[\big]{v_\lambda-v_\mu}^q_{L_q}\,ds
       \biggr)^{p/q}\\
    &\hspace{5em}\quad + (\lambda+\mu)^{p/q} \,
    \Bigl( 1 + \E\sup_{t \leq T} \norm[\big]{u_\lambda(t)}_{L_{2d+q-2}}^{p^*}
    + \E\sup_{t \leq T} \norm[\big]{u_\mu(t)}_{L_{2d+q-2}}^{p^*} \Bigr).
  \end{align*}
  It is immediately seen that, choosing $\varepsilon$ small enough and
  $\alpha$ large enough, we are left with
  \[
  \E\sup_{t\leq T} \norm[\big]{v_\lambda(t)-v_\mu(t)}_{L_q}^p
  \lesssim
  (\lambda+\mu)^{p/q} \,
  \Bigl( 1 + \E\sup_{t \leq T} \norm[\big]{u_\lambda(t)}_{L_{2d+q-2}}^{p^*}
  + \E\sup_{t \leq T} \norm[\big]{u_\mu(t)}_{L_{2d+q-2}}^{p^*} \Bigr),
  \]
  which implies, by the boundedness of $(u_\lambda)$ in
  $\H_{p^*}(L_{2d+q-2})$, that $(u_\lambda)$ is a Cauchy sequence in
  $\H_{p,\alpha}(L_q)$, hence also in $\H_p(L_q)$ by equivalence of
  (quasi-)norms.
\end{proof}

The strong convergence of $u_\lambda$ to a process $u \in \H_p(L_q)$
just established does not seem sufficient, unfortunately, to prove
that $u$ is a strict mild solution to \eqref{eq:0}. In fact, writing the
regularized equation \eqref{eq:reg} in its integral form
\begin{multline}
\label{eq:rm}
u_\lambda(t)
+ \int_0^t S(t-s)f_\lambda(u_\lambda(s))\,ds\\
= S(t)u_0 + \eta \int_0^t S(t-s)u_\lambda(s)\,ds
+ \int_0^t S(t-s)B(s,u_\lambda(s))\,dW(s),
\end{multline}
difficulties appear, as is natural to expect, when trying to pass to
the limit in the integral on the left-hand side. We are going to show
that boundedness assumptions on $(u_\lambda)$ in a smaller space imply
convergence of the term containing $f_\lambda(u_\lambda)$ in a
suitable norm, which is turn yields well-posedness in the strict mild
sense. First we state and prove a Lipschitz continuity result for
the solution map $u_0 \mapsto u$ of strict mild solution, which
immediately implies uniqueness.
\begin{lemma}     \label{lm:smile}
  Let $u_1$, $u_2$ be strict mild solutions in $\H_p(L_q)$ to
  \eqref{eq:0} with initial conditions $u_{01}$ and $u_{02}$,
  respectively. Then
  \[
  \norm{u_1-u_2}_{\H_p(L_q)} \lesssim 
  \norm{u_{01} - u_{02}}_{\L_p(L_q)}.
  \]
  In particular, if \eqref{eq:0} admits a strict mild solution
  $u \in \H_p(L_q)$, then it is unique and the solution map is
  Lipschitz continuous from $\L_p(L_q)$ to $\H_p(L_q)$.
\end{lemma}
\begin{proof}
  We use again an argument based on It\^o's formula and elementary
  inequalities. By definition of strict mild solution, we have $f(u_1)$,
  $f(u_2) \in L_1(L_q)$. Therefore, from
  \begin{multline*}
  d(u_1-u_2) + A(u_1-u_2)\,dt + \bigl(f(u_1)-f(u_2)\bigr)\,dt\\
  = \eta (u_1-u_2)\,dt + \bigl(B(u_1)-B(u_2)\bigr)\,dW,
  \qquad u_1(0)-u_2(0)=u_{01} - u_{02},
  \end{multline*}
  and Proposition \ref{prop:gf}, it follows
  \begin{align*}
  &\norm[\big]{v_1(t)-v_2(t)}_{L_{q}}^{q} 
  + q(\alpha-\eta) \int_0^t \norm{v_1-v_2}_{L_q}^q\,ds\\
  &\hspace{5em}\leq \norm[\big]{u_{01}-u_{02}}_{L_q}^q 
  + M(t)\\
  &\hspace{5em}\quad + \frac12 q(q-1) \int_0^t
  \norm[\big]{e^{-\alpha s}\bigl(B(u_1)-B(u_2)\bigr)}_{\gamma(H,L_q)}^2
           \norm[\big]{v_1-v_2}_{L_q}^{q-2}\,ds,
  \end{align*}
  where $v_i := e^{-\alpha\cdot} u_i$, $i=1,2$, and
  \[
  M(t) := \int_0^t e^{-\alpha s} \Phi'_{q}(v_1(s) - v_2(s))
  \bigl(B(u_1(s)) - B(u_2(s))\bigr)\,dW(s).
  \]
  By the Lipschitz continuity of $B$,
  \[
  \norm[\big]{e^{-\alpha s}\bigl(B(u_1)-B(u_2)\bigr)}_{\gamma(H,L_q)}^2
           \norm[\big]{v_1-v_2}_{L_q}^{q-2}
  \leq \norm[\big]{B}^2_{\lip} \norm[\big]{v_1-v_2}_{L_q}^{q},
  \]
  hence
  \begin{align*}
  &\norm[\big]{v_1(t)-v_2(t)}_{L_{q}}^{q}
  + q \Bigl( \alpha - \eta - \frac12 (q-1) \norm[\big]{B}^2_{\lip}
    \Bigr) \int_0^t \norm{v_1-v_2}_{L_q}^q\,ds\\
  &\hspace{5em} \leq \norm[\big]{u_{01}-u_{02}}_{L_q}^q + M(t),
  \end{align*}
  and we choose $\alpha$ so that
  $\bigl( \alpha - \eta - (q-1) \norm{B}^2_{\lip}/2\bigr)>0$.
  Taking suprema in time, raising to the power $p/q$, taking
  expectation, and raising to the power $1/p$, we get
  \begin{align*}
  &\norm[\big]{v_1-v_2}_{\H_p(L_q)} + q^{1/q} \Bigl( \alpha - \eta 
     - \frac12 (q-1) \norm[\big]{B}^2_{\lip} \Bigr)^{1/q}
  \norm[\big]{v_1-v_2}_{\L_p(L_q(L_q)))}\\
  &\hspace{5em} \lesssim \norm[\big]{u_{01}-u_{02}}_{\L_p(L_q)}
  + \norm[\big]{M_T^*}^{1/q}_{\L_{p/q}},
  \end{align*}
  where the implicit constant depends only on $p$ and $q$, and
  $M^*_T:=\sup_{t\leq T} \abs{M_t}$. The Burkholder-Davis-Gundy
  inequality yields
  \begin{align*}
  \norm[\big]{M_T^*}^{1/q}_{\L_{p/q}} &\lesssim 
  \norm[\big]{[M,M]_T^{1/2}}^{1/q}_{\L_{p/q}} =
  \norm[\big]{[M,M]_T^{1/{2q}}}_{\L_p}\\
  &= \norm[\bigg]{\biggl(%
  \int_0^T \norm[\big]{v_1-v_2}_{L_q}^{2(q-1)}%
  \norm[\big]{e^{-\alpha s}(B(u_1)-B(u_2))}^2_{\gamma(H,L_q)}
  \biggr)^{\frac{1}{2q}}}_{\L_p}\\
  &\leq \norm[\big]{B}^{2/q}_{\lip} \norm[\big]{v_1-v_2}_{\L_p(L_{2q}(L_q))}\\
  &\leq \varepsilon \norm[\big]{B}^{2/q}_{\lip} \norm[\big]{v_1-v_2}_{\H_p(L_q)}
  + N(\varepsilon) \norm[\big]{v_1-v_2}_{\L_p(L_{q}(L_q))},
  \end{align*}
  where we have used the Lipschitz continuity of $B$ and the inequality
  \[
  \norm{\phi}_{L_{2q}} \leq \norm{\phi}^{1/2}_{L_q} \norm{\phi}^{1/2}_{L_\infty}
  \leq \varepsilon \norm{\phi}_{L_\infty} + N(\varepsilon) \norm{\phi}_{L_q}
  \qquad \forall \phi \in L_q \cap L_\infty.
  \]
  We are thus left with
  \begin{align*}
  &\norm[\big]{v_1-v_2}_{\H_p(L_q)} + q^{1/q} \Bigl( \alpha - \eta - \frac12
     (q-1) \norm[\big]{B}^2_{\lip} \Bigr)^{1/q}
  \norm[\big]{v_1-v_2}_{\L_p(L_q(L_q)))}\\
  &\hspace{3em} \lesssim \norm[\big]{u_{01}-u_{02}}_{\L_p(L_q)}
  + \varepsilon \norm[\big]{B}^{2/q}_{\lip} \norm[\big]{v_1-v_2}_{\H_p(L_q)}
  + N(\varepsilon) \norm[\big]{v_1-v_2}_{\L_p(L_{q}(L_q))}.
  \end{align*}
  Since the implicit constant is independent of $\alpha$ and
  $\varepsilon$, this implies, upon choosing $\alpha$ large enough and
  $\varepsilon$ small enough, and recalling that the (quasi-)norms
  $\norm{\cdot}_{\H_{p,\alpha}(L_q)}$ and $\norm{\cdot}_{\H_p(L_q)}$ are
  equivalent,
  \[
  \norm[\big]{u_1-u_2}_{\H_p(L_q)} \eqsim
  \norm[\big]{u_1-u_2}_{\H_{p,\alpha}(L_q)} =
  \norm[\big]{v_1-v_2}_{\H_p(L_q)}\\
  \lesssim \norm[\big]{u_{01}-u_{02}}_{\L_p(L_q)}.
  \qedhere
  \]
\end{proof}

To prove uniqueness of solutions in $\H_p(L_q)$ we have used in a
crucial way the condition $f(u) \in L_1(L_q)$, which allows one to
apply Proposition \ref{prop:gf} (i.e. to use It\^o's formula). It is
thus natural to look for conditions ensuring weak compactness of
$f_\lambda(u_\lambda)$ in a functional space contained in
$\L_0(L_1(L_q))$. This is the motivation for the following
well-posedness result, conditional on boundedness of $(u_\lambda)$ in
a suitable norm.
\begin{prop}
  \label{prop:cwp}
  Let $p>0$, $q \geq 2$ and $p^* := p(2d+q-2)/q > d$. If the sequence
  $(u_\lambda)$ is bounded in $\H_{p^*}(L_{qd})$, then \eqref{eq:0}
  admits a unique strict mild solution $u \in \H_p(L_q)$ with
  continuous paths, and $u_0 \mapsto u$ is Lipschitz continuous from
  $\L_p(L_q)$ to $\H_p(L_q)$.
\end{prop}
\begin{proof}
  Since $d \geq 1$ implies $qd \geq 2d+q-2$ and
  $L_{qd} \embed L_{2d+q-2}$, it follows by Lemma \ref{lm:Cauchy} that
  $u_\lambda$ converges strongly to $u \in \H_p(L_q)$ as
  $\lambda \to 0$.  We are going to pass to the limit as
  $\lambda \to 0$ in mild form of equation \eqref{eq:reg}, i.e. in
  \eqref{eq:rm} above. Let us show that
  \[
  \int_0^t S(t-s)B(s,u_\lambda(s))\,dW(s) \xrightarrow{\lambda \to 0}
  \int_0^t S(t-s)B(s,u(s))\,dW(s)
  \]
  in probability for all $t \leq T$. In fact,\footnote{Note that here
    we are just using It\^o's isomorphism for the stochastic integral,
    not Burkholder's inequality, which does \emph{not} hold as the
    stochastic convolution is not, in general, a local martingale.}
  \begin{align*}
  &\E\norm[\bigg]{\int_0^t S(t-r)\bigl(B(u_\lambda(r)) -
    B(u(r))\,dW(r)\bigr)}_{L_q}^p\\
  &\hspace{5em} \lesssim \E\biggl( \int_0^t \norm[\big]{%
    S(t-r)\bigl(B(u_\lambda(r)) - B(u(r))}_{\gamma(H,L_q)}^2\,dr \biggr)^{p/2}\\
  &\hspace{5em} \lesssim_T \E\biggl( \int_0^t
    \norm[\big]{u_\lambda(r) - u(r)}^2_{L_q}\,dr \biggr)^{p/2}, 
  \end{align*}
  thanks to the ideal property of $\gamma(H,L_q)$, the contractivity
  of $S$, and the Lipschitz continuity of $B$. The last term tends to
  zero as $\lambda \to 0$ because $u_\lambda \to u$ in $\H_p(L_q)$
  and $\H_p(L_q) \embed \L_p(L_2(L_q))$.

  We can now consider the term in \eqref{eq:rm} involving
  $f_\lambda(u_\lambda)$. It follows from
  $\abs{f_\lambda} \leq \abs{f}$ and
  $\abs{f(x)} \lesssim 1 + \abs{x}^d$ for all $x \in \erre$ that, for
  any $s>1$,
  \[
  \norm[\big]{f_\lambda(u_\lambda)}_{\L_{p^*/d}(L_s(L_q))} \lesssim 
  1 + \norm[\big]{u_\lambda}^d_{\H_{p^*}(L_{qd})},
  \]
  so that $f_\lambda(u_\lambda)=f(J_\lambda u_\lambda)$ is bounded,
  hence weakly compact, in the reflexive Banach space
  $E:=\L_{p^*/d}(L_s(L_q))$ (recall that $p^*/d>1$ by assumption). In
  particular, there exists $g \in E$ and a subsequence of $\lambda$,
  denoted by the same symbol, such that $f(J_\lambda u_\lambda) \to g$
  weakly in $E$ as $\lambda \to 0$.  Since $J_\lambda u_\lambda \to u$
  strongly in $E$ as $\lambda \to 0$ and $f$, as an $m$-accretive
  operator on $E$, is also strongly-weakly closed thereon, we infer
  that $g \in f(u)$ $m$-a.e.. Since the linear operator
  \[
  \phi \mapsto \int_0^\cdot S(\cdot-s)\phi(s)\,ds
  \]
  is strongly (hence also weakly) continuous on $E$, we infer that
  \[
  \int_0^\cdot S(\cdot-s)f_\lambda(u_\lambda(s))\,ds \xrightarrow{\lambda \to 0}
  \int_0^\cdot S(\cdot-s) g(s)\,ds
  \]
  weakly in $E$, hence that
  \[
  u(t) = S(t)u_0 
  - \int_0^t S(t-s)\bigl(g(s) - \eta u(s)\bigr)\,ds
  + \int_0^t S(t-s)B(s,u(s))\,dW(s)
  \]
  for almost all $t \in [0,T]$. However, since $u$ admits a continuous
  $L_q$-valued modification, the identity must be satisfied for all
  $t \in [0,T]$. Existence is thus proved, and uniqueness as well as
  continuous dependence on the initial datum follow by the previous
  Lemma. The mild solution $u$, being a strong limit in $\H_p(L_q)$ of
  $(u_\lambda)$, inherits the path continuity of the latter.
\end{proof}

\subsection{A priori estimates}
As we have just seen, well-posedness in the strict mild sense in
$\H_p(L_q)$ for \eqref{eq:0} can be reduced to obtaining a priori
estimates for $(u_\lambda)$ in $\mathbb{H}_{p_1}(L_{q_1})$,
with $p_1>p$ and $q_1>q$ suitably chosen. 

\begin{prop}     \label{prop:apito}
  Let $p>0$ and $q \geq 2$. If $u_0 \in \L_p(\cF_0;L_q)$ and
  hypotheses $(\mathrm{A}_q)$, $(\mathrm{B}_{p,q})$ are satisfied,
  then there exists a constant $N$, independent of $\lambda$, such
  that
  \[
  \E\sup_{t \leq T} \norm[\big]{u_\lambda(t)}_{L_{q}}^p \leq N
  \Bigl( 1 + \E\norm[\big]{u_0}_{L_{q}}^p \Bigr).
  \]
\end{prop}
\begin{proof}
  The proof uses arguments analogous to ones already seen, hence we
  omit some detail. As in previous proofs, we begin observing that the
  regularized equation \eqref{eq:reg} admits a unique $L_q$-valued
  solution $u_\lambda$, and, setting
  $v_\lambda(t):=e^{-\alpha t} u_\lambda(t)$ for all $t \geq 0$, with
  $\alpha>\eta$ a constant to be fixed later, Proposition
  \ref{prop:gf} implies
  \begin{align*}
    &\norm[\big]{v_\lambda(t)}_{L_q}^{q} 
      + q(\alpha-\eta) \int_0^t \norm{v_\lambda}_{L_q}^q\,ds
    + \int_0^t e^{-\alpha s} \Phi'_{q}(v_\lambda)
         f_\lambda(u_\lambda)\,ds\\
    &\hspace{5em}\leq \int_0^t e^{-\alpha s} \Phi'_{q}(v_\lambda)
    B(u_\lambda)\,dW\\
    &\hspace{5em}\quad + \frac12 q(q-1) \int_0^t
    \norm[\big]{e^{-\alpha s}B(u_\lambda)}_{\gamma(H,L_q)}^2
    \norm[\big]{v_\lambda}_{L_q}^{q-2}\,ds.
  \end{align*}
  We shall denote the stochastic integral in the previous inequality by $M$.
  By the homogeneity of order $q-1$ of $\Phi_q'$, the monotonicity of
  $f_\lambda$, the inequality $\abs{f_\lambda} \leq \abs{f}$, the
  identity $\norm{\Phi_q'(x)}_{L_{q'}} = q \norm{x}_{L_q}^{q-1}$, and
  the elementary inequality $a^{q-1} \leq 1 + a^q$ for all $a \geq 0$,
  we have
  \begin{align*}
  e^{-\alpha s} \Phi'_{q}(v_\lambda) f_\lambda(u_\lambda)
  &= e^{-q\alpha s} \Phi'_{q}(u_\lambda)
    \bigl(f_\lambda(u_\lambda) - f_\lambda(0) + f_\lambda(0)\bigr)\\
  &\geq e^{-q\alpha s} \Phi'_{q}(u_\lambda) f_\lambda(0)
  \geq - e^{-q\alpha s} \Phi'_{q}(u_\lambda) \abs{f(0)}\\
  &\geq -q e^{-q\alpha s} \abs{f(0)} \norm{u_\lambda}_{L_q}^{q-1}\\
  &\geq -q e^{-q\alpha s} \abs{f(0)}
   - q e^{-q\alpha s} \abs{f(0)} \norm{u_\lambda}_{L_q}^q,
  \end{align*}
  hence
  \[
  \int_0^t e^{-\alpha s} \Phi'_{q}(v_\lambda)
  f_\lambda(u_\lambda)\,ds \geq - \frac{\abs{f(0)}}{\alpha} (1-e^{-q\alpha t})
  - q\abs{f(0)} \int_0^t \norm[\big]{v_\lambda}_{L_q}^q\,ds.
  \]
  Similarly, by the triangle inequality and Lipschitz continuity of
  $B$,
  \[
  \norm[\big]{e^{-\alpha s}B(u_\lambda)}_{\gamma(H,L_q)}^2 \leq
  2\norm[\big]{B}_{\lip}^2 \norm[\big]{v_\lambda}_{L_q}^2
  + 2e^{-2\alpha s} \norm[\big]{B(0)}_{\gamma(H,L_q)}^2,
  \]
  hence, thanks to the elementary inequality $a^{q-2} \leq 1 + a^q$,
  $a \geq 0$,
  \begin{align*}
  &\frac12 q(q-1) \int_0^t
      \norm[\big]{e^{-\alpha s}B(u_\lambda)}_{\gamma(H,L_q)}^2
      \norm[\big]{v_\lambda}_{L_q}^{q-2}\,ds\\
  &\hspace{5em} \leq q(q-1) \norm[\big]{B}_{\lip}^2
      \int_0^t \bigl(1 + e^{-2\alpha s} \norm[\big]{B(0)}_{\gamma(H,L_q)}^2\bigr)
      \norm[\big]{v_\lambda}_{L_q}^q\,ds\\
  &\hspace{5em}\quad + \frac{q(q-1)}{2\alpha} \norm[\big]{B(0)}_{\gamma(H,L_q)}^2
  (1 - e^{-2\alpha t}).
  \end{align*}
  We can thus write
  \begin{align*}
    &\norm[\big]{v_\lambda(t)}_{L_q}^{q} 
      + q\bigl(\alpha-\eta-\abs{f(0)} - (q-1)\norm[\big]{B}_{\lip}^2\bigr)
      \int_0^t \norm{v_\lambda}_{L_q}^q\,ds\\
    &\hspace{5em}\leq \frac{\abs{f(0)}}{\alpha}
      + \frac{q(q-1)}{2\alpha} \norm[\big]{B(0)}_{\gamma(H,L_q)}^2 + M_T^*,
  \end{align*}
  and we choose the constant $\alpha$ larger than
  $\eta + \abs{f(0)} + (q-1)\norm{B}_{\lip}^2$. Raising to the power
  $p/q$, taking suprema, then expectation, and taking the power $1/p$,
  we are left with
  \[
  \norm[\big]{v_\lambda}_{\H_p(L_q)} + N_1
  \norm[\big]{v_\lambda}_{\L_p(L_q(L_q))} \lesssim N_2 +
  \norm[\big]{M_T^*}_{\L_{p/q}}^{1/q},
  \]
  where
  \begin{align*}
  N_1 &:= q^{1/q}\Bigl(\alpha-\eta-\abs{f(0)} -
  (q-1)\norm[\big]{B}_{\lip}^2\Bigr)^{1/q},\\
  N_2 &:= \Bigl( \frac{\abs{f(0)}}{\alpha} +
    \frac{q(q-1)}{2\alpha} \norm[\big]{B(0)}_{\gamma(H,L_q)}^2
  \Bigr)^{1/q},
  \end{align*}
  and, by an argument based on the Burkholder-Davis-Gundy inequality
  and norm interpolation, as in the proof of Lemma \ref{lm:smile},
  \begin{align*}
  \norm[\big]{M_T^*}_{\L_{p/q}}^{1/q} &\lesssim
  \norm[\big]{[M,M]_T^{1/2}}_{\L_{p/q}}^{1/q} = 
  \norm[\big]{[M,M]_T^\frac{1}{2q}}_{\L_p}\\
  &\leq \varepsilon \norm[\big]{B}_{}^{2/q} \norm[\big]{v_\lambda}_{\H_p(L_q)}
  + N_3(\varepsilon) \norm[\big]{v_\lambda}_{\L_p(L_q(L_q))}.
  \end{align*}
  Collecting estimates yields
  \begin{align*}
  &\norm[\big]{v_\lambda}_{\H_p(L_q)} + N_1(\alpha)
  \norm[\big]{v_\lambda}_{\L_p(L_q(L_q))}\\
  &\hspace{5em} \leq N_4\Bigl(N_2 +
  \varepsilon \norm[\big]{B}_{\lip}^{2/q} \norm[\big]{v_\lambda}_{\H_p(L_q)}
  + N_3(\varepsilon) \norm[\big]{v_\lambda}_{\L_p(L_q(L_q))} \Bigr),
  \end{align*}
  for a constant $N_4$ independent of $\alpha$ and $\varepsilon$. The
  proof is completed choosing first $\varepsilon$ small enough, and
  then $\alpha$ large enough.
\end{proof}

\subsection{Proof of Theorem \ref{thm:main0}}
Since $u_0 \in \L_{p^*}(L_{qd})$ and hypotheses $(\mathrm{A}_{qd})$,
$(\mathrm{B}_{p^*,qd})$ are satisfied, the regularized equation
\eqref{eq:reg} admits a unique $L_{qd}$-valued (strict) mild solution
$u_\lambda$ for all $\lambda>0$. By Proposition \ref{prop:apito} the
sequence $(u_\lambda)$ is bounded in $\H_{p^*}(L_{qd})$, hence
Proposition \ref{prop:cwp} allows us to conclude that \eqref{eq:0}
admits a unique strict mild solution $u \in \H_p(L_q)$ and that the
solution map $u_0 \mapsto u$ is Lipschitz continuous from $\L_p(L_q)$
to $\H_p(L_q)$.\hfill$\square$


\section{Generalized solutions}
\label{sec:gen}
In this section we prove Theorems \ref{thm:ga} and \ref{thm:gm}. The
main tool is the Lipschitz continuity of the map $(u_0,B) \mapsto u$
established in the next Lemma. For reasons of notational compactness,
we set $L_{r,\alpha}(0,T;X):=L_r([0,T],\mu;X)$, where $\mu$ is the
measure on $[0,T]$ with density $t \mapsto e^{-r \alpha t}$.
\begin{lemma}     \label{lm:ga}
  Assume that $p>0$ and $q \geq 2$. Let $u_1$, $u_2 \in \H_p(L_q)$ be
  strict mild solutions to
  \[
  du_1 + Au_1\,dt + f(u_1)\,dt = \eta u_1\,dt + B_1\,dW, \qquad
  u_1(0)=u_{01},
  \]
  and
  \[
  du_2 + Au_2\,dt + f(u_2)\,dt = \eta u_2\,dt + B_2\,dW, \qquad
  u_2(0)=u_{02},
  \]
  respectively, where $B_1$, $B_2 \in \L_p(L_2(0,T;\gamma(H,L_q)))$
  and $u_{01}$, $u_{02} \in \L_p(L_q)$. Then, for any $\alpha>\eta$,
  \begin{equation}\label{eq:ga}
  \norm[\big]{u_1-u_2}_{\H_{p,\alpha}(L_q)}
  \lesssim \norm[\big]{u_{01}-u_{02}}_{\L_p(L_q)} +
  \norm[\big]{B_1-B_2}_{\L_p(L_{2,\alpha}(0,T;\gamma(H,L_q)))}.
  \end{equation}
  In particular,
  \begin{equation}\label{eq:gaa}
  \norm[\big]{u_1-u_2}_{\H_{p}(L_q)} 
  \lesssim \norm[\big]{u_{01}-u_{02}}_{\L_p(L_q)} +
  \norm[\big]{B_1-B_2}_{\L_p(L_{2}(0,T;\gamma(H,L_q)))}.    
  \end{equation}
  Moreover, the same estimates hold for generalized solutions.
\end{lemma}
\begin{proof}
  The proof uses again arguments analogous to those used in the proof
  of Lemma \ref{lm:smile}, therefore some detail will be omitted.

  Setting $v_i(t) := e^{-\alpha t} u_i(t)$, $i=1,2$, for all
  $t \geq 0$, with $\alpha > \eta$, it follows by Proposition
  \ref{prop:gf} and monotonicity of $f$,
  \begin{align*}
  &\norm[\big]{v_1(t)-v_2(t)}_{L_{q}}^{q} 
  + q(\alpha-\eta) \int_0^t \norm{v_1-v_2}_{L_q}^q\,ds\\
  &\hspace{5em} \leq \norm[\big]{u_{01}-u_{02}}_{L_q}^q 
  + \int_0^t e^{-\alpha s} \Phi'_{q}(v_1 - v_2)
    \bigl(B_1 - B_2\bigr)\,dW\\
  &\hspace{5em}\quad + \frac12 q(q-1) \int_0^t
  \norm[\big]{e^{-\alpha s}\bigl(B_1-B_2\bigr)}_{\gamma(H,L_q)}^2
           \norm[\big]{v_1-v_2}_{L_q}^{q-2}\,ds,
  \end{align*}
  where, by Young's inequality with exponents $q/(q-2)$ and $q/2$, the
  last term is estimated by
  \begin{align*}
  &\sup_{s\leq t} \norm[\big]{v_1(s)-v_2(s)}_{L_q}^{q-2}
  \int_0^t e^{-2\alpha s} \norm[\big]{B_1-B_2}_{\gamma(H,L_q)}^2\,ds\\
  &\hspace{5em} \leq
  \varepsilon \sup_{s\leq t} \norm[\big]{v_1(s)-v_2(s)}_{L_q}^{q}
  + N(\varepsilon) \biggl(
    \int_0^t e^{-2\alpha s} \norm[\big]{B_1-B_2}_{\gamma(H,L_q)}^2\,ds\biggr)^{q/2}.
  \end{align*}
  Choosing $\varepsilon$ smaller than one, hence, as before, raising to the
  power $p/q$, taking suprema in time, then expectation, and finally
  power $1/p$, we get
  \begin{align*}
  &\norm[\big]{v_1-v_2}_{\H_p(L_{q})} 
  + (\alpha - \eta)^{1/q} \norm[\big]{v_1-v_2}_{\L_p(L_q(L_{q}))}\\
  &\hspace{5em} \lesssim
  \norm[\big]{u_{01}-u_{02}}_{\L_p(L_q)}
  + \norm[\big]{B_1-B_2}_{\L_p(L_{2,\alpha}(0,T;\gamma(H,L_q)))}
  + \norm[\big]{M_T^*}_{\L_{p/q}}^{1/q},
  \end{align*}
  where $M$ denotes the stochastic integral with respect to $W$ in the
  first inequality above. Applying the Burkholder-Davis-Gundy
  inequality, elementary estimates, and Young's inequality with
  exponents $q/(q-1)$ and $q$, we have
  \begin{align*}
  \norm[\big]{M_T^*}_{\L_{p/q}}^{1/q} &\lesssim 
  \norm[\big]{[M,M]_T^{\frac{1}{2q}}}_{\L_p}\\
  &= \norm[\bigg]{\biggl( \int_0^T \norm[\big]{v_1-v_2}_{L_q}^{2(q-1)}%
     \norm[\big]{B_1-B_2}_{\gamma(H,L_q)}^2 e^{-2\alpha s} \,ds
     \biggr)^{\frac{1}{2q}}}_{\L_p}\\
  &\lesssim \varepsilon
  \norm[\big]{v_1-v_2}_{\H_p(L_{q})} 
  + N(\varepsilon) \norm[\big]{B_1-B_2}_{\L_p(L_{2,\alpha}(0,T;\gamma(H,L_q)))}.
  \end{align*}
  Choosing $\varepsilon$ suitably small, the last two inequalities yield
  \[
  \norm[\big]{v_1-v_2}_{\H_p(L_{q})}
  \lesssim
  \norm[\big]{u_{01}-u_{02}}_{\L_p(L_q)}
  + \norm[\big]{B_1-B_2}_{\L_p(L_{2,\alpha}(0,T;\gamma(H,L_q)))},
  \]
  which establishes the claim because
  \begin{align*}
  \norm[\big]{u_1-u_2}_{\H_{p,\alpha}(L_{q})} &= 
  \norm[\big]{v_1-v_2}_{\H_p(L_{q})},\\
  \norm[\big]{u_1-u_2}_{\L_p(L_{q,\alpha}(L_{q}))} &=
  \norm[\big]{v_1-v_2}_{\L_p(L_q(L_{q}))}.
  \end{align*}
  and equivalence of the norms in $\H_{p,\alpha}(L_q)$, $\alpha \geq 0$.
  It is easily seen that estimates \eqref{eq:ga} and \eqref{eq:gaa}
  are stable with respect to passage to the limit, hence the remain
  true for generalized solutions.
\end{proof}

\begin{proof}[Proof of Theorem \ref{thm:ga}]
  Let $p_1 \geq p$ be such that
  \[
  p_1^* := \frac{p_1}{q}(2d+q-2) > d.
  \]
  Note that $d \geq 1$ implies $p_1^* \geq p_1 \geq p$, hence
  $\L_{p_1^*}(L_{qd})$ is dense in $\L_p(L_q)$, so that there
  exists a sequence
  \[
  (u_{0n})_{n\in\mathbb{N}} \subset \L_{p_1^*}(L_{qd})
  \]
  such that $u_{0n} \to u_0$ in $\L_p(L_q)$ as $n \to \infty$.
  Similarly, recalling the isomorphism $\gamma(H,L_q) \simeq L_q(H)$,
  there also exists a sequence
  \[
  (B_n)_{n\in\mathbb{N}} \subset \L_{p_1^*}(L_2(0,T;\gamma(H,L_{qd})))
  \]
  such that $B_n \to B$ in $\L_p(L_2(0,T;\gamma(H,L_{q})))$ as
  $n \to \infty$. Then, by Theorem \ref{thm:main0}, for each
  $n \in \mathbb{N}$ the equation
  \[
  du_n(t) + Au_n(t)\,dt + f(u_n(t))\,dt = \eta u(t)\,dt +
  B_n(t)\,dW(t), \qquad u_n(0)=u_{0n},
  \]
  admits a unique strict mild solution
  $u_n \in \H_{p_1}(L_q) \embed \H_p(L_q)$, and the previous lemma
  yields
  \[
  \norm[\big]{u_n-u_m}_{\H_p(L_q)} \lesssim 
  \norm[\big]{u_{0n} - u_{0m}}_{\L_p(L_q)} 
  + \norm[\big]{B_n-B_m}_{\L_p(L_2(0,T;\gamma(H,L_q)))},
  \]
  hence $(u_n)$ is a Cauchy sequence in $\H_p(L_q)$. This implies that
  its strong limit $u \in \H_p(L_q)$ is a generalized solution to
  \eqref{eq:add}. Uniqueness and Lipschitz dependence on the initial
  datum follow immediately by \eqref{eq:ga}.
\end{proof}

\begin{proof}[Proof of Theorem \ref{thm:gm}]
  Let $w_1,w_2 \in \H_p(L_q)$ and consider the equation, for $i=1,2$,
  \[
  du_i(t) + Au_i(t)\,dt + f(u_i(t))\,dt = 
  \eta u_i(t)\,dt + B(w_i(t))\,dW(t), \qquad u_i(0)=u_{0i}.
  \]
  The assumptions on $B$ immediately imply that
  $B(w) \in \L_p(L_2(0,T;\gamma(H,L_q)))$ for all $w \in \H_p(L_q)$,
  hence the previous equation admits a unique generalized solution
  $u_i \in \H_p(L_q)$ by Theorem \ref{thm:ga}. In particular, the
  domain of the map $\Gamma$ is the whole $\H_p(L_q)$ and its image is
  contained in $\H_p(L_q)$. We are now going to show that $\Gamma$ is
  a contraction in $\H_{p,\alpha}(L_q)$ for small $T$. In fact,
  inequality \eqref{eq:ga} yields
  \[
  \norm[\big]{u_1-u_2}_{\H_{p,\alpha}(L_q)}
  \lesssim
  \norm[\big]{u_{01}-u_{02}}_{\L_p(L_q)} +
  \norm[\big]{B(w_1)-B(w_2)}_{\L_p(L_{2,\alpha}(0,T;\gamma(H,L_q)))},
  \]
  where, by the Lipschitz continuity of $B$,
  \begin{align*}
  &\norm[\big]{B(w_1)-B(w_2)}^p_{\L_p(L_{2,\alpha}(0,T;\gamma(H,L_q)))}\\
  &\hspace{5em} = \E\biggl(
  \int_0^T e^{-2\alpha s} \norm[\big]{B(w_1)-B(w_2)}^2_{\gamma(H,L_q)}\,ds
     \biggr)^{p/2}\\
  &\hspace{5em} \leq \norm[\big]{B}_{\lip}^p
  \E\biggl( \int_0^T e^{-2\alpha s} \norm[\big]{w_1 - w_2}^2_{L_q}\,ds
  \biggr)^{p/2}\\
  &\hspace{5em} \leq \norm[\big]{B}_{\lip}^p T^{p/2}
  \norm[\big]{w_1 - w_2}^p_{\H_{p,\alpha}(L_q)}.
  \end{align*}
  This implies that $\Gamma$ is a contraction on $\H_{p,\alpha}(L_q)$
  for $T$ small enough, hence that a unique generalized solution
  exists that depends Lipschitz continuously on the initial datum.  By
  a classical patching procedure, the result can be extended to
  arbitrary finite $T$.
\end{proof}


\section{Mild solutions}
\label{sec:mild}
In this section we prove Theorem \ref{thm:mild}. The proof is split
into two parts: first we prove existence, showing that one can pass to
the limit in the term of \eqref{eq:rm} containing
$f_\lambda(u_\lambda)$ in the weak topology of $\L_1(L_1(L_1))$. Then
we prove uniqueness, as a consequence of continuous dependence on the
initial datum, via an extension of Proposition \ref{prop:gf}. We
proceed this way because, as will be apparent soon, the symmetry
condition on $F$ and the ``regularizing'' assumptions on $A$ are needed
only to prove uniqueness.

\subsection{Existence}
We shall use the following weak convergence criterion (see
\cite[Theorem~18]{Bre-mm}).
\begin{lemma}
\label{lm:Brezis}
  Let $(Y,\mathcal{A},\mu)$ be a finite measure space. Assume that
  $\gamma$ is a maximal monotone graph in $\erre \times \erre$ with
  $\dom(\gamma)=\erre$ and $0 \in \gamma(0)$. If the sequences of functions
  $(z_n)$, $(y_n) \subset L_0(Y,\mathcal{A},\mu)$ indexed by
  $n\in\enne$ are such that $\lim_{n\to\infty} y_n = y$ $\mu$-a.e.,
  $z_n \in \gamma(y_n)$ $\mu$-a.e. for all $n$, and there exists a constant
  $N$ such that
  \[
  \int_Y z_ny_n\,d\mu < N \qquad \forall n \in \enne,
  \]
  then there exist $z \in L_1(Y,\mu)$ and a subsequence $(n_k)_k$
  such that $z_{n_k} \to z$ weakly in $L_1(Y,\mu)$ as $k \to \infty$
  and $z \in \gamma(y)$ $\mu$-a.e.
\end{lemma}
\begin{proof}[Sketch of proof]
  The weak compactness in $L_1(Y,\mu)$ of $(z_n)$ is a consequence of
  $z_ny_n = G(y_n) + G^*(z_n)$, where $G$ is a convex function with
  $G(0)=0$ such that $\gamma=\partial G$. In fact,
  $\dom(\gamma)=\erre$ implies that $G^*$ is superlinear at infinity
  (see~{\S}\ref{ssec:conv}), which in turn implies, by the criterion
  of de la Vall\'ee-Poussin (see e.g.~\cite[Theorem~4.5.9]{Bog-MT1}),
  that $(z_n)$ is uniformly integrable in $L_1(Y,\mu)$, hence, by the
  Dunford-Pettis theorem, it is relatively weakly compact thereon (see
  e.g.~\cite[Corollary~4.7.19]{Bog-MT1}). The fact that
  $z\in \gamma(y)$ $\mu$-a.e. requires a further (short) argument
  based on monotonicity (see~\cite{Bre-mm} for details).
\end{proof}

\smallskip

Let us also recall some further notation. For $q \geq 2$, define the
homeomorphism $\phi_q$ of $\erre$ and the maximal monotone graph
$\tilde{f}$ in $\erre \times \erre$ as
\[
\phi_q: x \mapsto x \abs{x}^{q-2} \equiv \abs{x}^{q-1} \operatorname{sgn} x,
\qquad
\tilde{f} := f \circ \phi_q^{-1}.
\]
Since $0 \in \tilde{f}(0)$, there exists a convex function
$\tilde{F}:\erre \to \erre$ with $\tilde{F}(0)=0$ such that
$\partial \tilde{F}=\tilde{f}$. As usual, we shall denote by
$\tilde{F}^*$ the convex conjugate of $\tilde{F}$. We recall that
$\tilde{F}^*$ is convex and superlinear at infinity, because
$\tilde{f}$ is finite on the whole real line (see~{\S}\ref{ssec:conv}).

In the next statement $g$ stands for the process defined in Definition
\ref{def:mild}.
\begin{prop}
\label{prop:mild}
  Let $p \geq q \geq 2$ and $0 \in f(0)$. Assume that
  \begin{itemize}
  \item[\emph{(a)}] $u_0 \in \L_p(L_q)$;
  \item[\emph{(b)}] hypothesis $(\mathrm{A}_s)$ is satisfied for
    $s \in \{1, q, qd\}$;
  \item[\emph{(c)}] hypothesis $(\mathrm{B}_{p,q})$ is satisfied.
  \end{itemize}
  Then there exists a mild solution $u \in \H_p(L_q)$ to
  \eqref{eq:0}. Moreover, $u$ has continuous paths and satisfies
  $\hat{F}(u)$, $\tilde{F}^*(g) \in \L_1(L_1(L_1))$.
\end{prop}
\begin{proof}
  We proceed in several steps. 
  \smallskip\par\noindent
  \noindent\textsc{Step 1.} We begin showing that the
  generalized solution to \eqref{eq:0}, which exists and is unique
  thanks to Theorem \ref{thm:gm}, can be approximated by strict mild
  solutions to suitable equations. Let $u$ be the generalized solution
  to \eqref{eq:0} and $\delta>0$. Then there exists $n_0 \in \enne$
  such that $\norm{u-u_n}_{\H_p(L_q)} < \delta$ for all $n > n_0$,
  where $u_n$ is the unique generalized solution to
  \[
  du_n + Au_n\,dt + f(u_n)\,dt = \eta u_n\,dt + B(u_{n-1})\,dW, \qquad
  u_n(0)=u_0.
  \]
  In turn, for any $n>n_0$, there exists $\nu=\nu(n)$ and $u_{0\nu}$,
  $\bigl[B(u_{n-1})\bigr]_\nu$ such that
  \[
  \norm{u_n-u^\nu_n}_{\H_p(L_q)} \lesssim \norm{u_0-u_{0\nu}}_{\L_p(L_q)}
  + \norm[\big]{[B(u_{n-1})]_\nu - B(u_{n-1})}_{\L_p(L_2(0,T;\gamma(H,L_q)))}
  < \delta,
  \]
  where $u_n^\nu$ is the unique strict mild solution to
  \[
  du_n^\nu + Au_n^\nu\,dt + f(u_n^\nu)\,dt = \eta u_n^\nu\,dt +
  \bigl[B(u_{n-1})\bigr]_\nu\,dW,
  \qquad u^\nu_n(0) = u_{0\nu}.
  \]
  In particular, by the triangle inequality, one can construct a
  sequence of strict mild solutions $\bar{u}_n:=u_n^{\nu(n)}$ such that
  $\norm{u-\bar{u}_n}_{\H_p(L_q)}$ is less than (a constant times)
  $\delta$, i.e., being $\delta$ arbitrary, the sequence
  $\bar{u}_n:=u_n^{\nu(n)}$ converges to $u$ in $\H_p(L_q)$.
  \smallskip\par\noindent
  \textsc{Step 2.}\label{lm:apfel} We are now going to show that there exists a
  constant $N$, independent of $n$, such that
  \[
  \E\int_0^T \Phi'_q(\bar{u}_n(s)) \bar{g}_n(s)\,ds < N,
  \]
  where $\bar{g}_n \in f(\bar{u}_n)$ $m$-a.e. is the selection of
  $f(\bar{u}_n)$ appearing in the definition of strict mild solution.
  Let $\bar{B}_n:=[B(u_{n-1})]_{\nu(n)}$ and
  $v_n(t) := e^{-\alpha t} \bar{u}_n(t)$ for all $t \in [0,T]$, with
  $\alpha>\eta$ a constant. Proposition \ref{prop:gf} and obvious
  estimates yield
  \begin{align*}
  \norm[\big]{v_n(t)}_{L_\infty(0,T;L_q)}^q 
  &+ q(\alpha-\eta) \int_0^T \norm{v_n(s)}_{L_q}^q\,ds
  + \int_0^T e^{-\alpha s} \Phi'_q(v_n(s)) \bar{g}_n(s)\,ds\\
  &\lesssim \norm[\big]{u_{0n}}_{L_q}^q 
  + \sup_{t \leq T} \abs[\bigg]{%
    \int_0^t e^{-\alpha s} \Phi'_{q}(v_n(s)) \bar{B}_n(s)\,dW}\\
  &\quad + \frac12 q(q-1) \int_0^T
  \norm[\big]{e^{-\alpha s} \bar{B}_n(s)}_{\gamma(H,L_q)}^2
           \norm[\big]{v_n(s)}_{L_q}^{q-2}\,ds,
  \end{align*}
  where $M$, as before, stands for the stochastic integral above.  By
  Young's inequality with exponents $q/(q-2)$ and $q/2$, the
  last term can be estimated by
  \[
  \varepsilon \norm[\big]{v_n}_{L_\infty(0,T;L_q)}^q
  + N(\varepsilon) \norm[\big]{\bar{B}_n}^q_{L_2(0,T;\gamma(H,L_q))},
  \]
  so that, choosing $\varepsilon$ small enough, raising to the power
  $p/q$, and taking expectation, we obtain
  \begin{multline*}
  \norm[\big]{v_n}^p_{\H_p(L_q)} 
  + \E\biggl( \int_0^T e^{-\alpha s} \Phi'_q(v_n(s)) \bar{g}_n(s)\,ds
      \biggr)^{p/q}\\
  \lesssim \norm[\big]{u_{0n}}_{\L_p(L_q)}^p + \E(M_T^*)^{p/q}
  + \norm[\big]{\bar{B}_n}^p_{\L_p(L_2(0,T;\gamma(H,L_q)))}.
  \end{multline*}
  By an argument already used before, based on the
  Burkholder-Davis-Gundy inequality and Young's inequality with
  exponents $q/(q-1)$ and $q$, we have
  \[
  \E(M_T^*)^{p/q} \lesssim \varepsilon \norm[\big]{v_n}^p_{\H_p(L_q)}
  + N(\varepsilon) \norm[\big]{\bar{B}_n}^p_{\L_p(L_2(0,T;\gamma(H,L_q)))}
  \]
  thus also, choosing $\varepsilon$ small enough,
  \[
  \E\biggl( \int_0^T e^{-\alpha s} \Phi'_q(v_n(s)) \bar{g}_n(s)\,ds 
    \biggr)^{p/q}
  \lesssim \norm[\big]{u_{0n}}_{\L_p(L_q)}^p
  + \norm[\big]{\bar{B}_n}^p_{\L_p(L_2(0,T;\gamma(H,L_q)))},
  \]
  where the first term on the right-hand side is bounded because, by
  definition of generalized solution, $u_{0n}$ converges to $u_0$ in
  $\L_p(L_q)$. Moreover, denoting the norm of 
  $\L_p(L_2(0,T;\gamma(H,L_q)))$ by $\norm{\cdot}$ for
  simplicity,
  \[
  \norm[\big]{\bar{B}_n - B(u_{n-1})} = 
  \norm[\big]{[B(u_{n-1})]_{\nu(n)} - B(u_{n-1})}< \delta/2, 
  \]
  hence $\norm{\bar{B}_n} < \norm{B(u_{n-1})} + \delta/2$, where, by
  the Lipschitz continuity of $B$,
  \begin{align*}
  \norm{B(u_{n-1})} &\leq \norm{B(u_{n-1}) - B(0)} + \norm{B(0)}\\
  &\lesssim \norm{u_{n-1}}_{\H_p(L_q)} + \norm{B(0)}.
  \end{align*}
  Since $u_n \to u$ in $\H_p(L_q)$ as $n \to \infty$, it follows that
  $\norm[\big]{u_n}_{\H_p(L_q)}$ is bounded, which in turn implies
  that $\norm[\big]{\bar{B}_n}$ is also bounded.  We conclude that
  there exists a constant $N$, independent of $n$, such that
  \[
  \E\biggl( \int_0^T e^{-\alpha s} \Phi'_q(v_n(s)) \bar{g}_n(s)\,ds 
    \biggr)^{p/q} < N.
  \]
  Since $p/q \geq 1$, it follows by Jensen's inequality that
  \[
  \E\int_0^T e^{-\alpha s} \Phi'_q(v_n(s)) \bar{g}_n(s)\,ds < N
  \]
  (where $N$ might differ from the previous one). The proof is
  finished observing that
  $e^{-\alpha s} \Phi'_q(v_n(s)) = e^{-q\alpha s}
  \Phi'_q(\bar{u}_n(s))$
  and $e^{-q\alpha s} \geq e^{-q\alpha T}$ for all $s \in [0,T]$.
  \smallskip\par\noindent
  \textsc{Step 3.} We are going to show that the uniform estimate of
  the previous step implies that the generalized solution $u$ is also
  a mild solution to \eqref{eq:0}.

  The conclusion of the previous step can equivalently be written as
  \[
  \int_\Xi \phi_q(\bar{u}_n) \bar{g}_n \,dm < N,
  \]
  where $\Xi:=\Omega \times [0,T] \times D$, and $N$ is a constant
  independent of $n$.  Since $\phi_q$ is a homeomorphism of $\erre$,
  setting $v_n:=\phi_q(\bar{u}_n)$ and recalling the definition
  $\tilde{f}:= f \circ \phi_q^{-1}$, we have (see
  e.g. \cite[p.~II.12]{Bbk:Ens} about the associativity of composition of
  graphs)
  \[
  \bar{g}_n \in f(\bar{u}_n) = f \circ \phi_q^{-1}
  \bigl( \phi_q(\bar{u}_n) \bigr)
  = \tilde{f}(v_n),
  \]
  hence the previous estimate can be written as
  \[
  \int_{\Xi} v_n \bar{g}_n\,dm < N,
  \]
  where $\bar{g}_n \in \tilde{f}(v_n)$ $m$-a.e.  Since, by continuity of
  $\phi_q$, $v_n = \phi_q(u_n) \to \phi_q(u) = :v$ $m$-a.e. as
  $n \to \infty$ and $\dom\bigl(f \circ \phi_q^{-1}\bigr)=\erre$,
  Lemma \ref{lm:Brezis} implies that there exists $g \in L_1(m)$ and a
  subsequence $(n')$ of $(n)$ such that $\bar{g}_{n'} \to g$ weakly in
  $L_1(m)$ as $n' \to \infty$, and
  \[
  g \in f \circ \phi_q^{-1}(v) = f \circ \phi_q^{-1}(\phi_q(u)) = f(u)
  \quad m\text{-a.e.}
  \]
  Since $-A$ generates a $C_0$-semigroup of contractions on $L_1(D)$
  by assumption, one obtains
  \[
  \int_0^t S(t-s) \tilde{g}_n(s)\,ds \to \int_0^t S(t-s) g(s)\,ds
  \]
  weakly in $\L_1(L_1(L_1))$ as $\lambda \to 0$, by a reasoning
  completely analogous to that used in the last part of the proof of
  Proposition \ref{prop:cwp}. Similarly, convergence of the stochastic
  convolutions follows as in the proof just mentioned because
  \[
  \norm[\big]{\bar{B}_n - B(u)}_{\L_p(L_2(0,T;\gamma(H,L_q)))}
  \xrightarrow{n \to \infty} 0.
  \]
%
%
  \textsc{Step 4.} It remains to prove that $\hat{F}(u)$,
  $\tilde{F}^*(g) \in \L_1(L_1(L_1))$. In fact, we have
  $\bar{g}_nv_n = \tilde{F}(v_n) + \tilde{F}^*(\tilde{g}_n)$ because
  $g_n \in \tilde{f}(v_n) = \partial \tilde{F}(v_n)$, hence, by the
  previous Step, there exists a constant $N$, independent of $n$, such
  that
  \[
  \int_\Xi \tilde{F}(v_n) \,dm < N, \qquad
  \int_\Xi \tilde{F}^*(g_n) \,dm < N.
  \]
  The convexity of $\tilde{F}$ and $\tilde{F}^*$ implies the weak
  lower semicontinuity in $L_1(m)$ of
  \[
  \phi \mapsto \int_\Xi \tilde{F}(\phi)\,dm,
  \qquad \phi \mapsto \int_\Xi \tilde{F}^*(\phi)\,dm,
  \]
  hence $\tilde{F}^*(g) \in L_1(m)$ because $\bar{g}_n \to g$ weakly in
  $L_1(m)$. Moreover, $v_n \to v$ in $m$-measure and
  \[
  \norm{v_n}_{L_1(m)} = \norm{\bar{u}_n}_{L_{q-1}(m)} \to
  \norm{u_n}_{L_{q-1}(m)} = \norm{v}_{L_1(m)},
  \]
  by virtue of the strong convergence $\bar{u}_n \to u$ in $\H_p(L_q)$
  and the embedding $\H_p(L_q) \embed L_{q-1}(m)$. We thus have
  $v_n \to v$ strongly in $L_1(m)$, hence, similarly as above,
  $\hat{F}(u) = \tilde{F}(v) \in L_1(m)$.
\end{proof}

\begin{rmk}
  If $p<q$, the above proof does not work because Jensen's inequality
  reverses. However, a weaker integrability result can still be
  obtained. Namely, again by Jensen's inequality, we have
  \[
  \int_\Xi \abs[\big]{f_n \phi_q(u_n)}^{p/q} \,d\mu< N,
  \]
  uniformly over $n$, where the constant $N$ depends also on the
  Lebesgue measure of $D$. Setting
  $x^{\langle a \rangle} := \abs{x}^a\,\operatorname{sgn} x$ for all
  $x \in \erre$ and $a>0$, taking into account that $0 \in f(0)$,
  the previous estimate can equivalently be written as
  \[
  \int_\Xi f_n^{\langle p/q \rangle} \phi_q^{\langle p/q \rangle}(u_n) \,d\mu 
  < N.
  \]
  For any $a>0$ the function $x \mapsto x^{\langle a \rangle}$ is a
  homeomorphism of $\erre$, hence the function
  $\psi_{p,q}: x \mapsto \phi_q^{\langle p/q \rangle}(x)$ is also a
  homeomorphism of $\erre$. We clearly have
  \[
  f_n^{\langle p/q \rangle} \in f^{\langle p/q \rangle} \circ
  \psi_{p,q}^{-1} \bigl( \psi_{p,q}(u_n) \bigr)
  \]
  $\mu$-a.e., hence there exists $z \in L_1(\mu)$ such that
  $f_{n_k}^{\langle p/q \rangle} \to z$ weakly in $L_1(\mu)$ along a
  subsequence $(n_k)$, with
  $z \in f^{\langle p/q \rangle} \circ \psi_{p,q}^{-1} \bigl(
  \psi_{p,q}(u) \bigr) = f^{\langle p/q \rangle}(u)$.
  We thus have, for a generalized solution, that $\zeta \in f(u)$ is
  only in $L_{p/q}(\mu)$, rather than in $L_1(\mu)$. This in
  particular implies that it does not seem possible any longer to
  claim that $u$ is a mild solution, even in a very weak sense, as the
  semigroup $S$ is not defined in $L_q(D)$ spaces with $0<q<1$.
\end{rmk}

\subsection{Uniqueness}
The aim of this subsection is to prove continuous dependence on the
initial datum (from which uniqueness follows immediately) for mild
solutions to \eqref{eq:0}, \emph{without} assuming that
$g \in L_1(L_q)$. We need to assume, however, the same integrability
conditions on the solution that are established in the proof of
Proposition \ref{prop:mild}, as well as positivity and regularizing
properties for the resolvent of $A$ and a symmetry condition on $F$.

The key is the following estimate for the difference of two mild
solutions to \eqref{eq:0}, whose proof is inspired by an analogous
result, in a different setting, in \cite{BDPR-porous}.
\begin{lemma}
  \label{lm:dfcl}
  Under the hypotheses of Theorem \ref{thm:mild}, assume that $u_i$,
  $i=1,2$, satisfies
  \[
  u_i(t) + \int_0^t S(t-s) \bigl(g_i(s) - \eta u_i(s)\bigr)\,ds
  = S(t)u_{0i} + \int_0^t S(t-s) B(u_i(s))\,dW(s)
  \]
  for all $t \in [0,T]$, where $g_i \in \L_1(L_1(L_1))$, $g_i \in f(u_i)$
  $m$-a.e., and $\hat{F}(u_i)$, $\tilde{F}^*(g_i) \in
  \L_1(L_1(L_1))$.
  Then, setting $v_i(t):=e^{-\alpha t} u_i(t)$, $t \in [0,T]$, for
  $\alpha \geq 0$ constant, one has
  \begin{align*}
    &\norm[\big]{v_1(t)-v_2(t)}_{L_q}^q
      + q(\alpha-\eta) \int_0^t \norm[\big]{v_1(s)-v_2(s)}_{L_q}^q\,ds\\
    &\hspace{3em} + \int_0^t e^{-\alpha s}\Phi_q'(v_1(s)-v_2(s))
                    (g_1(s)-g_2(s))\,ds\\
    &\hspace{3em} \leq \norm[\big]{u_{01}-u_{02}}_{L_q}^q\\
    &\hspace{3em} \quad + \int_0^t e^{-\alpha s} \Phi_q'(v_1(s)-v_2(s))
      \bigl(B(u_1(s)) - B(u_2(s))\bigr)\,dW(s)    \\
    &\hspace{3em} \quad + \frac12 q(q-1) \int_0^t
      \norm[\big]{e^{-\alpha s}\bigl(B(u_1(s)) - B(u_2(s))\bigr)}^2_{\gamma(H,L_q)}
      \norm[\big]{v_1(s)-v_2(s)}_{L_q}^{q-2}\,ds
  \end{align*}
  for all $t \in [0,T]$.
\end{lemma}
\begin{proof}
  Given $\sigma \in \enne$ such that $(I + \varepsilon A)^{-\sigma}$
  maps $L_1(D)$ to $L_q(D)$, set
  $\ep{h}_i := (I + \varepsilon A)^{-\sigma}h$ for all
  $h \in \{u_i,u_{0i},g_i,v_i\}$, and
  $\ep{B}_i := (I+\varepsilon)^{-\sigma} B(u_i)$.  Then $g_i \in L_1(L_q)$
  and $\ep{v}_i$ is the unique $L_q$-valued strong solution to
  \[
  d\ep{v}_i + A\ep{v}_i\,dt + (\alpha - \eta)\ep{v}_i\,dt 
  + e^{-\alpha t} \ep{g}_i\,dt = e^{-\alpha t} \ep{B}_i\,dW,
  \qquad \ep{u}_i(0)=\ep{u}_{0i}.
  \]
  It\^o's formula and \eqref{eq:tr-ineq} then yield
  \begin{equation}
  \label{eq:dfcl}
  \begin{split}
  &\norm[\big]{\ep{v}_1(t)-\ep{v}_2(t)}_{L_q}^q
    + q(\alpha-\eta) \int_0^t \norm[\big]{\ep{v}_1(s)-\ep{v}_2(s)}_{L_q}^q\,ds\\
  &\hspace{3em} + \int_0^t e^{-\alpha s}\Phi_q'(\ep{v}_1(s)-\ep{v}_2(s))
                  (\ep{g}_1(s)-\ep{g}_2(s))\,ds\\
  &\hspace{3em} \leq \norm[\big]{\ep{u}_{01}-\ep{u}_{02}}_{L_q}^q\\
  &\hspace{3em} \quad + \int_0^t e^{-\alpha s} \Phi_q'(\ep{v}_1(s)-\ep{v}_2(s))
    \bigl(\ep{B}_1(s) - \ep{B}_2(s)\bigr)\,dW(s)\\
  &\hspace{3em} \quad + \frac12 q(q-1) \int_0^t
    \norm[\big]{e^{-\alpha s}\bigl(\ep{B}_1(s) - \ep{B}_2(s)\bigr)}^2_{\gamma(H,L_q)}
    \norm[\big]{\ep{v}_1(s)-\ep{v}_2(s)}_{L_q}^{q-2}\,ds
  \end{split}
  \end{equation}
  We are now going to pass to the limit as $\varepsilon \to 0$ in the
  above inequality. Since ${(I+\varepsilon A)^{-\sigma}}$ converges to
  the identity in $\cL(L_q)$ in the strong operator topology, it
  immediately follows that
  \begin{align*}
  \norm[\big]{\ep{v}_1(t) - \ep{v}_2(t)}_{L_q} 
  &\xrightarrow{\varepsilon \to 0}
  \norm[\big]{v_1(t) - v_2(t)}_{L_q} \qquad \forall t \in [0,T],\\
  \norm[\big]{\ep{u}_{01}-\ep{u}_{02}}_{L_q} 
  &\xrightarrow{\varepsilon \to 0} \norm[\big]{u_{01} - u_{02}}_{L_q}.
  \end{align*}
  Since $(I+\varepsilon A)^{-\sigma}$ is contracting in $L_q$,
  $\norm{\ep{v}_1-\ep{v}_2}_{L_q} \leq \norm{v_1}_{L_q} +
  \norm{v_2}_{L_q}$ pointwise,
  hence, by Fubini's theorem, $v \in \H_p(L_q)$, and the dominated
  convergence theorem,
  \[
  \int_0^t \norm[\big]{\ep{v}(s)}_{L_q}^q\,ds 
  \xrightarrow{\varepsilon \to 0}
  \int_0^t \norm[\big]{v(s)}_{L_q}^q\,ds \qquad \forall t \leq T.
  \]
  The dominated convergence theorem also immediately shows that
  \begin{align*}
  &\int_0^t
  \norm[\big]{e^{-\alpha s}\bigl(\ep{B}_1(s) - \ep{B}_2(s)\bigr)}^2_{\gamma(H,L_q)}
  \norm[\big]{\ep{v}_1(s)-\ep{v}_2(s)}_{L_q}^{q-2}\,ds\\
  &\hspace{5em} \xrightarrow{\varepsilon \to 0}
  \int_0^t
  \norm[\big]{e^{-\alpha s}\bigl(B(u_1(s)) - B(u_2(s))\bigr)}^2_{\gamma(H,L_q)}
  \norm[\big]{v_1(s) - v_2(s)}_{L_q}^{q-2}\,ds.
  \end{align*}
  Let us now consider the last term on the left-hand side of
  \eqref{eq:dfcl}. Recalling the definition of the homeomorphism
  $\phi_q: x \mapsto x\abs{x}^{q-2}$, we have
  \begin{align*}
  &\int_0^t e^{-\alpha s}\Phi_q'(\ep{v}_1(s)-\ep{v}_2(s))
                  (\ep{g}_1(s)-\ep{g}_2(s))\,ds\\
  &\hspace{5em} = 
  q \int_0^t e^{-q\alpha s} \Phi_q'\bigl(\ep{u}_1(s) - \ep{u}_2(s)\bigr)
     \bigl(\ep{g}_1(s) - \ep{g}_2(s) \bigr)\,ds\\
  &\hspace{5em} \eqsim 
  \int_0^t\!\int_D \bigl(\ep{g}_1(s) - \ep{g}_2(s)\bigr)
     \phi_q\bigl(\ep{u}_1(s) - \ep{u}_2(s)\bigr)\,dx\,ds.
  \end{align*}
  The properties of $(I+\varepsilon A)^{-\sigma}$ imply easily that
  $\ep{g}_i \to g_i$ in $\L_1(L_1(L_1))$, hence in $m$-measure, as
  $\varepsilon \to 0$. Similarly, since $\ep{u}_i \to u_i$ in
  $m$-measure and $\phi_q$ is continuous,
  $\phi_q(\ep{u}_1-\ep{u}_2) \to \phi_q(u_1-u_2)$ in $m$-measure. In
  particular,
  \[
  \bigl(\ep{g}_1 - \ep{g}_2\bigr) \phi_q\bigl(\ep{u}_1 - \ep{u}_2\bigr)
  \xrightarrow{\varepsilon \to 0} 
  (g_1 - g_2) \phi_q(u_1 - u_2)
  \]
  in $m$-measure. We are going to show that this convergence takes
  place in $\L_1(L_1(L_1))$.  To this purpose, it suffices to show, by
  Vitali's theorem, that the sequence on the right-hand side is
  uniformly integrable (UI).
  Let $\delta \in \mathopen]0,1/2\mathclose]$ be arbitrary but
  fixed. By Young's inequality with conjugate functions $\tilde{F}$
  and $\tilde{F}^*$ and the definition
  $\hat{F}:=\tilde{F} \circ \phi_q$,
  \begin{align*}
  \abs[\big]{\bigl(\ep{g}_1 - \ep{g}_2\bigr)%
      \phi_q\bigl(\ep{u}_1 - \ep{u}_2\bigr)} &\eqsim 
  \abs[\big]{\delta(\ep{g}_1 - \ep{g}_2)} \,
  \abs[\big]{\phi_q\bigl(\delta(\ep{u}_1 - \ep{u}_2)\bigr)}\\
  &\leq \hat{F}\bigl(\abs[\big]{\delta(\ep{u}_1 - \ep{u}_2)} \bigr)
  + \tilde{F}^*\bigl(\abs[\big]{\delta(\ep{g}_1 - \ep{g}_2)} \bigr)\\
  &= \hat{F}\bigl(\delta(\ep{u}_1 - \ep{u}_2) \bigr)
  + \tilde{F}^*\bigl(\delta(\ep{g}_1 - \ep{g}_2) \bigr).
  \end{align*}
  In the last step we have used that $\hat{F}$ and $\tilde{F}^*$ are
  even: in fact, since $F$ is even and $F(0)=0$, we infer that $f$ is
  odd, $\tilde{f}$ is odd, hence $\tilde{F}$, $\tilde{F}^*$ and
  $\hat{F}$ are even with
  $\tilde{F}(0) = \tilde{F}^*(0) = \hat{F}(0) = 0$. Then it follows
  that $\tilde{F}^*$ and $\hat{F}$ are increasing on $\erre_+$ (this
  can also be seen by
  $\partial \tilde{F}^* = \tilde{f}^{-1} = \phi_q \circ f^{-1} \geq 0$
  and $\partial \hat{F} = f \phi_q' \geq 0$ on $\erre_+$). Therefore
  $\hat{F}(cx) = \hat{F}(c\abs{x}) \leq \hat{F}(\abs{x}) = \hat{F}(x)$
  for all $x \in \erre$ and $c\in[0,1]$, and the same holds for
  $\tilde{F}^*$. In particular,
  \begin{align*}
  \hat{F}\bigl(\delta(\ep{u}_1 - \ep{u}_2) \bigr)
  &= \hat{F}\Bigl( \frac12 (2\delta \ep{u}_1) + \frac12 (-2\delta \ep{u}_2)
  \Bigr)\\
  &\leq \frac12 \hat{F}\bigl( 2\delta \ep{u}_1 \bigr)
        + \frac12 \hat{F}\bigl( 2\delta \ep{u}_2 \bigr)\\
  &\leq \frac12\bigl( \hat{F}(\ep{u}_1) + \hat{F}(\ep{u}_2) \bigr),
  \end{align*}
  and, completely analogously,
  \[
  \tilde{F}^*\bigl(\delta(\ep{g}_1 - \ep{g}_2) \bigr) \leq 
  \frac12 \bigl( \tilde{F}^*(\ep{g}_1) + \tilde{F}^*(\ep{g}_2) \bigr),
  \]
  thus also
  \[
  \abs[\big]{\bigl(\ep{g}_1 - \ep{g}_2\bigr)%
  \phi_q\bigl(\ep{u}_1 - \ep{u}_2\bigr)}
  \lesssim \hat{F}(\ep{u}_1) + \hat{F}(\ep{u}_2)
  + \tilde{F}^*(\ep{g}_1) + \tilde{F}^*(\ep{g}_2).
  \]
  Let us now observe that, by Jensen's inequality for positive
  operators (see e.g. \cite{Haa07}),
  \begin{align*}
  \tilde{F}^*(\ep{g}_i) = \tilde{F}^*\bigl( (I+\varepsilon A)^{-\sigma} g_i \bigr)
  &\leq (I+\varepsilon A)^{-\sigma} \tilde{F}^*(g_i)\\
  \hat{F}^*(\ep{u}_i) = \hat{F}\bigl( (I+\varepsilon A)^{-\sigma} u_i \bigr)
  &\leq (I+\varepsilon A)^{-\sigma} \hat{F}(u_i)
  \end{align*}
  But since $\hat{F}(u_i)$, $\tilde{F}^*(g_i) \in \L_1(L_1(L_1))$ by
  assumption, hence
  $(I+\varepsilon A)^{-\sigma}\hat{F}(u_i) \to \hat{F}(u_i)$ and
  $(I+\varepsilon A)^{-\sigma}\tilde{F}^*(g_i) \to \tilde{F}^*(g_i)$
  as $\varepsilon \to 0$ in $\L_1L_1L_1$, it follows that the sequence
  $\abs[\big]{(\ep{g}_1 - \ep{g}_2) \phi_q(\ep{u}_1 - \ep{u}_2)}$ is
  dominated by a convergent sequence of $\L_1(L_1(L_1))$, which is a
  fortiori UI.  Then
  $(\ep{g}_1 - \ep{g}_2) \phi_q(\ep{u}_1 - \ep{u}_2)$ is also UI,
  because a (positive) sequence dominated by a UI sequence is itself
  UI. We have thus proved that the last term on the left-hand side of
  \eqref{eq:dfcl} converges in probability for all $t \in [0,T]$ to
  \[
  \int_0^t e^{-\alpha s}\Phi_q'(v_1(s)-v_2(s)) (g_1(s)-g_2(s))\,ds.
  \]
  It remains only to consider the stochastic integral on the
  right-hand side of \eqref{eq:dfcl}, which converges to
  \[
  \int_0^t e^{-\alpha s} \Phi_q'(v_1(s)-v_2(s))
  \bigl(B(u_1(s)) - B(u_2(s))\bigr)\,dW(s)
  \]
  in probability for all $t \in [0,T]$ as $\varepsilon \to 0$. The
  proof is based on an argument entirely analogous to the one already used
  in the proof of Proposition \ref{prop:gf}, and is hence omitted.
\end{proof}

We can now prove uniqueness of mild solution to \eqref{eq:0} and their
continuous dependence on the initial datum.
\begin{prop}
  Under the hypotheses of Theorem \ref{thm:mild}, assume that
  $u \in \H_p(L_q)$ is a mild solution to \eqref{eq:0}. Then $u$ is
  the unique mild solution such that
  $\hat{F}(u) + \tilde{F}^*g \in \L_1(L_1(L_1))$.  Moreover, the
  solution map $u_0 \mapsto u$ is Lipschitz continuous from
  $\L_p(L_q)$ to $\H_p(L_q)$.
\end{prop}
\begin{proof}
  Let $u_1$, $u_2$ be as in the previous Lemma, with $u_{01}$,
  $u_{0,2} \in \L_p(L_q)$. Then
  \begin{align*}
    &\norm[\big]{v_1(t)-v_2(t)}_{L_q}^q
      + q(\alpha-\eta) \int_0^t \norm[\big]{v_1-v_2}_{L_q}^q\,ds
    + \int_0^t e^{-\alpha s}\Phi_q'(v_1-v_2)
                    (g_1-g_2)\,ds\\
    &\hspace{3em} \leq \norm[\big]{u_{01}-u_{02}}_{L_q}^q
    + \int_0^t e^{-\alpha s} \Phi_q'(v_1-v_2)
      \bigl(B(u_1) - B(u_2)\bigr)\,dW    \\
    &\hspace{3em} \quad + \frac12 q(q-1) \int_0^t
      \norm[\big]{e^{-\alpha s}\bigl(B(u_1) - B(u_2)\bigr)}^2_{\gamma(H,L_q)}
      \norm[\big]{v_1-v_2}_{L_q}^{q-2}\,ds,
  \end{align*}
  where
  \[
  \Phi_q'(v_1-v_2)\bigl(g_1 - g_2\bigr) = q e^{-(q-1)\alpha \cdot}
  \ip[\big]{g_1-g_2}{\phi_q(u_1-u_2)} \geq 0.
  \]
  We are now in the condition to use exactly the same proof of Lemma
  \ref{lm:smile}, arriving at
  \[
  \norm[\big]{u_1-u_2}_{\H_p(L_q)} \lesssim
  \norm[\big]{u_{01}-u_{02}}_{\L_p(L_q)},
  \]
  which proves that $u_0 \mapsto u \in \lip(\H_p(L_q),\L_p(L_q))$ and,
  as an immediate consequence, uniqueness of the solution.
\end{proof}

\begin{rmk}
  It is clear by the previous proof that we do \emph{not} have
  well-posedness in the space $\H_p(L_q)$, as our uniqueness result
  holds only under additional assumptions on the solution itself. The
  problem of unconditional uniqueness in $\H_p(L_q)$ remains therefore
  open.
\end{rmk}


\bibliographystyle{amsplain}
\bibliography{ref}

\end{document}